\documentclass[12pt,draft]{article}
\usepackage{amsmath,amsfonts,amsthm,amssymb,amscd,colordvi}
\newcommand{\p}{\partial}

\newcommand{\vp}{\varphi}
\newcommand{\vt}{\vartheta}
\newcommand{\thee}{\theta^-_1}
\newcommand{\nnu}{\Theta}

\newcommand{\vv}{{\bf v}}
\newcommand{\vN}{{{}^N\!v}}
\newcommand{\Vv}{{\bf V}}
\newcommand{\RRR}{{\bf R}}
\newcommand{\va}{{v'_P}^N}
\newcommand{\vb}{{v'_P}^\nu}
\newcommand{\vc}{w_P^\nu}
\newcommand{\vcc}{{\widetilde w}^{\nu N}}

\newcommand{\bb}{\mbox{\boldmath$\beta$}}

\newcommand{\R}{{\mathbb R}}

\newcommand{\IP}{{\bf P}}
\newcommand{\Z}{{\mathbb Z}}
\newcommand{\E}{{\bf E}}

\newcommand{\T}{{\mathbb T}}
\newcommand{\N}{{\mathbb N}}
\newcommand{\PP}{{\bf P}}

\newcommand{\hD}{\widehat\Delta}

\newcommand{\aA}{{\cal A}}

\newcommand{\DD}{{\cal D}}
\newcommand{\FF}{{\cal F}}

\newcommand{\cH}{{\cal H}}
\newcommand{\LL}{{\cal L}}
\newcommand{\MM}{{\cal M}}

\newcommand{\RR}{{\cal R}}
\newcommand{\KK}{{\cal K}}
\newcommand{\TT}{{\cal T}}

\newcommand{\Sig}{{\cal S}}

\newcommand{\strela}{\rightharpoonup}

\newcommand{\Op}{\mathop{\rm Op}\nolimits}

\def\12{\tfrac12}

\def\lan{\langle}
\def\ran{\rangle}

\theoremstyle{plain}
\newtheorem{theorem}{Theorem}[section]
\newtheorem{lemma}[theorem]{Lemma}
\newtheorem{proposition}[theorem]{Proposition}
\newtheorem{corollary}[theorem]{Corollary}
\theoremstyle{definition}
\newtheorem{definition}[theorem]{Definition}
\theoremstyle{remark}

\numberwithin{equation}{section}

\begin{document}
\author{Sergei B. Kuksin}
\title{Damped-driven KdV and effective equation for long-time behaviour of its
solutions.}
\date{}
\maketitle

\begin{abstract}
For the damped-driven KdV equation
$$
\dot u-\nu{u_{xx}}+u_{xxx}-6uu_x=\sqrt\nu\,\eta(t,x),\; x\in S^1,\
\int u\,dx\equiv \int\eta\,dx\equiv0\,,
$$
with $0<\nu\le1$ and smooth in $x$
white in $t$ random force $\eta$, we study the limiting long-time
 behaviour
 of the KdV integrals of motions $(I_1,I_2,\dots)$,
 evaluated along a solution $u^\nu(t,x)$, as $\nu\to0$. We prove that
 for $0\le\tau:= \nu t \lesssim1$ the vector
$
I^\nu( \tau)=(I_1(u^\nu( \tau,\cdot)),I_2(u^\nu( \tau,\cdot)),\dots),
$
 converges in distribution to a limiting
process $I^0(\tau)=(I^0_1,I^0_2,\dots)$. The $j$-th component $I_j^0$ equals
$\12(v_j(\tau)^2+v_{-j}(\tau)^2)$, where
$v(\tau)=(v_1(\tau), v_{-1}(\tau),v_2(\tau),\dots)$ is the vector of Fourier coefficients
of a solution of an
{\it effective equation} for the dam\-ped-driven KdV. This new equation
 is a quasilinear stochastic
heat equation with a non-local nonlinearity,
written in the Fourier coefficients. It is well posed.
\end{abstract}

\tableofcontents

\section{Introduction}\label{s_intr}
In this work we continue the study of  randomly perturbed and damped KdV
equation, commenced in \cite{KP08}. Namely, we consider the  equation
\begin{equation}\label{kdv}
u_t-\nu u_{xx}+u_{xxx}- 6 uu_x=\sqrt{\nu} \,\eta(t,x),
\end{equation}
where
$x\in S^1\ \mathop{=}\limits^{\rm def}\ \mathbb R/2\pi\mathbb Z$, $\int_{S^1}u\,dx=0$,
 and $\nu>0$ is a small positive parameter.
The random stationary force $\eta=\eta(t,x)$ is
$\
\eta=\frac{d}{d t}\left(\sum_{s\in\mathbb Z_0}b_s\beta_s(t) e_s(x)\right)
$.
Here $\mathbb Z_0=\mathbb Z\setminus\{0\}$, $\beta_s$ are standard independent
 Wiener processes defined on a probability space $(\Omega,{\cal F},{\bf P})$, and
$\{e_s,\ s\in\mathbb Z_0\}$ is the usual trigonometric basis
\begin{equation*}
e_s(x)=\begin{cases}\displaystyle \frac{1}{\sqrt{\pi}}\,\cos(sx),&\quad s>0,
 \\[2mm]
\displaystyle \frac{1}{\sqrt{\pi}}\,\sin(sx),&\quad s<0.
\end{cases}
\end{equation*}
The coefficients $\nu$ and $\sqrt{\nu}$ in (\ref{kdv}) are balanced in such a way
 that solutions of the equation stays of order one as $t\to\infty$ and $\nu\to0$,
see \cite{KP08}. The coefficients $b_s$  are non-zero and are even in $s$,
 i.e. $b_s=b_{-s}\ne0$ $\forall\,s\ge1$. When  $|s|\to\infty$ they  decay faster
 than any negative power of $|s|$: for any $m\in \mathbb Z^+$ there is $C_m>0$
 such that
\begin{equation*}
|b_s|\le C_m|s|^{-m} \qquad \hbox{for all }s\in\mathbb Z_0.
\end{equation*}
This implies that the force $\eta(t,x)$ is smooth in $x$ for any $t$.
We study  behaviour of solutions for (\ref{kdv}) with given smooth initial
data
\begin{equation}\label{k0}
u(0,x)=u_0(x)\in C^\infty(S^1)
\end{equation}
for
\begin{equation}\label{k1}
    0\le t\le \nu^{-1}  T,\qquad 0<\nu\ll1.
\end{equation}
 Here $T$ is any fixed positive constant.

The KdV equation \eqref{kdv}${}_{\nu=0}$ is integrable. That is to say, the function
space $\{u(x): \int u\,dx=0\}$ admits analytic symplectic coordinates
$v=(\vv_1,\vv_2,\dots)=\Psi(u(\cdot))$,
 where $\vv_j=(v_j,v_{-j})^t\in\R^2$,
 such that the quantities $I_j=\12|\vv_j|^2$, $j\ge1$, are actions
 (integrals of motion),
while $\vp_j=\,$Arg$\,\vv_j$, $j\ge1$, are  angles. In the $(I,\vp)$-variables KdV
takes the integrable  form
\begin{equation}\label{k2}
    \dot I=0,\qquad \dot\vp=W(I),
\end{equation}
where $W(I)\in\R^\infty$ is the {\it frequency vector}, see Section~1.2. 
\footnote{The actions $I$ and the angles $q$ were constructed first (before the
Cartesian coordinates $v$), starting with the pioneer works by Novikov and
Lax in 1970's. See in \cite{McKT76, ZMNP, K2, KaP}.
}
 The
integrating map $\Psi$ is called the {\it nonlinear Fourier transform}.
\footnote{ The reason is that
an analogy of $\Psi$, a map which integrates the linearised KdV equation
$\ \dot u+u_{xxx}=0$, is the usual Fourier transform.}

We are mostly concerned with behaviour of actions $I(u(t))\in\R^\infty$ of solutions
for the perturbed KdV equation
 \eqref{kdv} for $t$, satisfying \eqref{k1}. For this end let us  write equations
for $I(v)$ and $\vp(v)$, using the slow time $\tau=\nu t\in[0,T]$:
\begin{equation}\label{k3}
    dI(\tau)=F(I,\vp)\,d\tau +\sigma(I,\vp)\,d\beta(\tau),\qquad
    d\vp =\nu^{-1}W(I)\,d\tau+\dots,
\end{equation}
where the dots stand for terms of order one, $\beta=(\beta_1,\beta_2,\dots)^t$ and
$\sigma(I,\vp)$ is an infinite matrix. For finite-dimensional
stochastic systems of the form \eqref{k3} under certain non-degeneracy assumptions,
for the $I$-component of solutions for \eqref{k3}  the averaging principle holds. That
is, when $\nu\to0$ the $I$-component of a solution
 converges in distribution to a solution of the averaged equation
\begin{equation}\label{k33}
dI=\lan F\ran(I)\,d\tau + \lan\sigma\ran(I)\,d\beta(\tau).
\end{equation}
Here $\lan F\ran$ is  the averaged drift,   $\lan F\ran=\int F(I,\vp)\,d\vp$,
 and the dispersion
matrix $\lan\sigma\ran$ is a square root of the averaged diffusion
$\int \sigma(I,\vp)\sigma^t(I,\vp)\,d\vp$. This result was claimed in \cite{Khas68} and
was first proved in \cite{FW03}; see \cite{Kif04} for recent development.
 In \cite{KP08} we established ``half" of this result
for solutions of eq.~\eqref{k33} which corresponds to  \eqref{kdv}.
 Namely, we have shown that for solutions $u_\nu(\tau,x)$ of \eqref{kdv},
\eqref{k0}, where $t=\nu^{-1}\tau$ and $0<\tau\le T$,

\noindent{\bf
i)}\, the set of laws of actions $\{\DD I(u_\nu(\tau))\}$ is tight in the space of continuous
 trajectories
$I(\tau)\in h^p_I$, $0\le\tau\le T$, where the space $h^p_I$ is given the norm
$\
|I|_{h^p_I}=2\sum_{j=1}^\infty j^{1+2p}|I_j|
$
and  $p$ is any number $\ge3$;
\smallskip

\noindent{\bf
ii)}\,  any limiting measure $\lim_{\nu_j\to0}\DD I(u_{\nu_j}(\cdot))$ is a law of a weak
solution $I^0(\tau)$ of eq.~\eqref{k33} with the initial condition
\begin{equation}\label{k4}
    I(0)=I_0:=I(u_0).
\end{equation}
The  solutions $I^0(\tau)$ are {\it regular} in the sense that all
 moments of the random variables $\sup_{0\le\tau\le T}|I^0(\tau)|_{h_I^r}$, $r\ge0$,
are finite.
\smallskip

Similar results are obtained in \cite{KP08} for limits (as $\nu_j\to 0$)  of stationary
 in time solutions for eq.~\eqref{kdv}.
\smallskip

If we knew that \eqref{k33}, \eqref{k4} has a unique solution $I^0(\tau)$,
 then ii) would imply that
\begin{equation}\label{k5}
\DD I(u_\nu(\cdot))\strela \DD I^0(\cdot)\quad{\rm as}\quad \nu\to0,
\end{equation}
as in the finite-dimensional case. But the uniqueness is far from obvious since
 \eqref{k33} is
a bad equation in the bad phase-space $\R^\infty_+$: the dispersion $\lan \sigma\ran$
is not Lipschitz in $I$, and the drift
$\lan F\ran(I)$ is an unbounded operator. In this paper we show that still the
convergence \eqref{k5} holds true:
\smallskip

{\bf Theorem A.} The problem \eqref{k33}, \eqref{k4} has a solution $I^0(\tau)$ such that
the convergence \eqref{k5} holds.
\smallskip

The proof of this result, given in   Section~4, Theorem~\ref{t_final}, relies on
a new construction, crucial for  this work. Namely, it turns
out that the `bad' equation \eqref{k33} may be lifted to a `good'
 {\it effective equation} on the
variable
 $v=(\vv_1,\vv_2,\dots)$, $\vv_j\in\R^2$,
which transforms to \eqref{k33} under the mapping
$$
\pi_I:v\mapsto I,\qquad I_j=\12|\vv_j|^2.
$$
 To derive the effective equation  we evoke the mapping
 $\Psi$ to  transform
eq.~\eqref{kdv}, written in the slow time $\tau$, to a system of  stochastic
equations on the vector $v(\tau)$
\begin{equation}\label{k6}
d\vv_k(\tau)=\nu^{-1}d\Psi_k(v)V(u)\,d\tau +P_k(v)\,d\tau
+\sum_{j\ge1} B_{kj}(v)\,d\bb_j(\tau),\quad k\ge1.
\end{equation}
Here $V(u)=-u_{xxx}+6uu_x$ is the vector-field of KdV, $P_k\,d\tau +\sum B_{kj}\,d\bb_j$ is the
perturbation $u_{xx}+\eta(\tau,x)$,  written in the $v$-variables,
 and $\bb_j$'s are standard Wiener processes in $\R^2$ (so $B_{kj}$'s are $2\times2$-blocks). We will refer to the system \eqref{k6} as to the {\it $v$-equations}.

 The system  \eqref{k6} is singular when $\nu\to0$.
The  effective equations for  \eqref{k6} is a system of
 regular stochastic equations
\begin{equation}\label{k7}
d\vv_k(\tau) = \lan P\ran_k\,d\tau + \lan\lan B\ran\ran_{kj}(v)\,d\beta_j(\tau),
\quad k\ge1.
\end{equation}
To define the {\it effective drift} $\lan P\ran$ and the {\it effective dispersion}
$\lan\lan B\ran\ran$, for any $\theta\in\T^\infty$ let us denote by $\Phi_\theta$ the
linear operator in the  space of sequences $v=(\vv_1,\vv_2,\dots)$ which
rotates each two-vector
$\vv_j$ by the angle $\theta_j$. The rotations $\Phi_\theta$ act on vector-fields
on the $v$-space, and $\lan P\ran$ is the result of the action of  $\Phi_\theta$ on $P$,
averaged in $\theta$:
\begin{equation}\label{eff}
\lan P\ran(v)=\int_{\T^\infty}\Phi_{-\theta}P(\Phi_\theta v)\,d\theta
\end{equation}
($d\theta$ is the Haar measure on $\T^\infty$).


Consider the diffusion operator $BB^t(v)$ for the $v$-equations \eqref{k6}. It defines a (1,1)-tensor on the linear space of vectors $v$. The averaging of this tensor
with respect to the transformations $\Phi_\theta$ is a tensor,  corresponding
 to the  operator
\begin{equation}\label{AvDiff}
\lan BB^t\ran(v)=\int_{\T^\infty}\Phi_{-\theta}\cdot \big((BB^t)(\Phi_\theta v)\big)\cdot \Phi_\theta \,d\theta.
\end{equation}
This is the {\it averaged diffusion operator}.
The effective dispersion operator  $ \lan\lan B\ran\ran(v)$ is its non-symmet\-ric square root:
\begin{equation}\label{effD}
\lan\lan B\ran\ran(v)\cdot\lan\lan B\ran\ran^t(v)=\lan BB^t\ran(v).
\end{equation}
Such a square  root is non-unique. The one, chosen in this work, is
given by an explicit construction and is analytic in $v$ (while the {\it symmetric}
square root of $\lan BB^t\ran(v)$ is only a H\"older-$\12$ continuous
function of $v$). See Sections~1.5 and 2.

Let us provide the space of vectors $v$ with the norms $|\cdot|_r$, $r\ge0$, where
$|v|^2_r=\sum_j|\vv_j|^2j^{1+2r}$. A solution of eq.~\eqref{k7} is called {\it regular} if
all moments of all random variables $\sup_{0\le\tau\le T}|v(\tau)|_r$, $r\ge0$, are finite.
\smallskip

\noindent{\bf Theorem B.}
System \eqref{k7} has at most one regular strong solution $v(\tau)$ such that $v(0)=\Psi(u_0)$.
\smallskip

This result is proved in Section~4, where we show that system \eqref{k7} is a quasilinear
stochastic heat equation, written in Fourier coefficients.

The effective system \eqref{k7} is useful to study  eq.~\eqref{kdv} since this is
 a lifting of the averaged
equations \eqref{k33}. The corresponding result, stated below, is proved in Section~3:
\smallskip

\noindent{\bf Theorem C.}
For every weak solution $I^0(\tau)$ of \eqref{k33} as in assertion ii) there exists a regular
weak solution $v(\tau)$ of \eqref{k7} such that $v(0)=\Psi(u_0)$ and
 $\DD\big(\pi_I(v(\cdot))\big)=
\DD(I^0(\cdot))$. Other way round, if $v(\tau)$ is a regular weak solution of \eqref{eff},
then $I(\tau)=\pi_I (v(\tau))$ is a weak solution of \eqref{k33}.
\smallskip

We do not know if a regular weak solution of problem \eqref{k33}, \eqref{k4} is unique.
 But from Theorem~B
we know that a regular weak solution of the Cauchy problem for the effective
 equation \eqref{k7} is
unique, and through Theorem~C it implies uniqueness of a solution for
 \eqref{k33}, \eqref{k4}  as in item~ii). This proves Theorem~A.

In Section 5 we evoke  some intermediate results from \cite{KP08} to show that after
averaging in $\tau$  distribution of the actions of a solution $u^\nu$ for \eqref{kdv}
become asymptotically (as $\nu\to0$)
independent from distribution of the angles, and the angles become uniformly distributed on
the torus $\T^\infty$. In particular, for  any continuous function $f\ge0$ such that $\int f=1$, we
have
$$
\int_0^Tf(\tau)\DD\vp(u^\nu(\tau))\,d\tau\strela d\theta\qquad{\rm as}\quad \nu\to0.
$$
\medskip

The recipe \eqref{eff} allows to construct effective equations for other perturbations of KdV,
with or without randomness. These are non-local nonlinear equations with interesting
properties. In particular, if the perturbation is given by a Hamiltonian nonlinearity
$\nu(\p/\p x)f(u,x)$,  then the effective system is Hamiltonian and integrable
(its hamiltonian depends only on the actions $I$).

The effective equations \eqref{k7} are instrumental to study other problems, related to
eq.~\eqref{kdv}. In particular, they may be used to prove the convergence \eqref{k5} when
$u_\nu(\tau)$ are stationary solutions of \eqref{kdv} and $I^0(\tau)$ is a stationary solution
for \eqref{k33}. See \cite{K10} for discussion of these and some related results; the proof will be
published elsewhere. We are certain that corresponding effective equations may be used
to study other perturbations of KdV, including the damped equation \eqref{kdv}${}_{\eta=0}$.

Our results are related to the Whitham averaging for perturbed KdV, see Appendix.

\medskip
\noindent {\it Agreements.} Analyticity of maps $B_1\to B_2$ between
Banach spaces $B_1$ and $B_2$, which are the real parts of complex
spaces $ B_1^c$ and $B_2^c$, is understood in the sense of
Fr\'echet. All analytic maps which we consider possess the following
additional property: for any $R$ a map analytically extends to a
complex $(\delta_R>0)$--neighbourhood of the ball $\{|u|_{B_1}<R\}$
in $B_1^c$. Such maps are Lipschitz on bounded subsets of $B_1$.
  When a property of a random variable holds
 almost sure, we often drop the specification ``a.s.".  All metric spaces are provided
 with the Borel sigma-algebras.   All sigma-algebras which
 we consider in this work are assumed to be completed with respect to the
 corresponding probabilities.

\noindent {\it  Notations.}  $\chi_A$ stands for the indicator
function of a set $A$ (equal 1 in $A$ and vanishing outside it).
By $\DD\xi$ we denote the distribution (i.e. the law) of a random variable
$\xi$. For a measurable set $Q\subset\R^n$ we denote by $|Q|$ its
Lebesgue measure.

\medskip\par
\noindent{\it Acknowledgments.} I wish to thank for discussions and advises
B.~Dubrovin, F.~Flandoli, N.~V.~Krylov, Y.~Le~Jan, R.~Liptser, S.~P.~Novikov and
B.~Tsirelson.
I am especially obliged to A.~Piatnitski for explaining me some  results, related to the
constructions in Section~1.4, and for critical remarks on a preliminary version of this
work.

\section{Preliminaries}\label{s_prel}

Solutions of problem \eqref{kdv}, \eqref{k0}
satisfy uniform in $t$ and $\nu$ a priori estimates (see \cite{KP08}):
\begin{equation}\label{a_est1}
{\mathbb E}\big\{\exp\big(\sigma\|u(t)\|_0^2\big)\big\}\le c_0,\quad
\quad{\mathbb E}\big(\|u(t)\|_m^k\big)\le c_{m,k},
\end{equation}
%
%
for any $m,k\ge 0$ and any $\sigma\le (2\max b_s^2)^{-1}$. Here $\|\cdot\|_m$
stands for the norm in the Sobolev space $H^m=\{u\in H^m(S^1):\int u\,dx=0\}$,
$\ \|u\|_m^2=\int (\p^m u/\p x^m)^2\,dx.$
  To study further properties
of solutions for \eqref{kdv} with small $\nu$ we need the nonlinear Fourier transform
$\Psi$ which integrates the KdV equation.

\subsection{Nonlinear Fourier transform for KdV}

For $s\ge0$ denote by $h^s$ the Hilbert space, formed by the vectors\\
$v=(v_1,v_{-1}, v_2, v_{-2},\dots)$ and provided with  the weighted
 $l_2$-norm $|\cdot|_s$,
$$
|v|_s^2=\sum\limits_{j=1}^\infty j^{1+2s}(v_j^2+v_{-j}^2).
$$
We set ${\bf v}_j=\left(\begin{array}{c}\!\!v_j\!\! \\ \!\!v_{-j}\!\!\end{array}
\right)$,
$j\in\mathbb Z^+=\{j\ge1\}$, and will also write vectors $v$ as $v=(\vv_1,\vv_2,
\dots)$.
For any $v\in h^s$ we define the vector of actions
$\
I(v)=(I_1,I_2,\dots),\quad I_j=\frac12|\vv_j|^2.
$
Clearly $I\in h^s_{I+}\subset h^s_I$. Here $h^s_I$ is the weighted $l^1$-space,
$$
h^s_I=\big\{I\,:\, |I|_{h^s_I}=2\sum\limits_{j=1}^\infty j^{1+2s}|I_j|<\infty\big\},
$$
and $h^s_{I+}$ is the positive octant $h^s_{I+}=\{h\in h^s_I: I_j\ge0\ \forall\,j\}$.

\begin{theorem}\label{t_kp08}
There exists an analytic diffeomorphism $\Psi\,:\,H^0\mapsto h^0$ and an
analytic functional $K$ on $h^0$ of the form
$K(v)=\widetilde K(I(v))$, where the function $\widetilde K(I)$ is analytic on
the space $h^0_{I+}$,
with the following properties
\begin{enumerate}
\item  The mapping $\Psi$ defines, for any $m\in\mathbb Z^+$, an analytic diffeomorphism
$\Psi\,:\,H^m\mapsto h^m$;
\item  The map $d\Psi(0)$ takes the form $\sum u_s e_s\mapsto v$, $v_s=|s|^{-1/2}u_s$;
\item A curve $u\in C^1(0,T;H^0)$ is a solution of the KdV equation
 (\ref{kdv})${}_{\nu=0}$
if and only if $v(t)=\Psi(u(t))$ satisfies the equation
\begin{equation}\label{BNF}
\dot{{\bf v}}_j=\left(\begin{array}{cc}0&-1\\1&0\end{array}
\right)\frac{\partial\widetilde K}{\partial I_j}(I){\bf v}_j,\qquad j\in\mathbb Z^+.
\end{equation}
\item For $m=0,1,2,\dots$ there are polynomials $P_m$ and $Q_m$ such that
$$
|d^j\Psi(u)|_m\le P_m(\|u\|_m),\quad \|d^j(\Psi^{-1}(v))\|_m\le Q_m(|v|_m),\ \
j=0,1,2,
$$
for all $u$ and $v$ and all $m\ge 0$.
\end{enumerate}
\end{theorem}

See  \cite{KaP} for items 1-3 and \cite{KP08} for item 4.
 The coordinates $v=\Psi(u)$ are called  the
 {\it  Birkhoff coordinates} and the form \eqref{BNF} of KdV --
its {\it  Birkhoff normal form}.

The analysis in Section~4 requires the following amplification of
 Theorem~\ref{t_kp08}, stating that the nonlinear Fourier transform $\Psi$
 ``is quasilinear'':

\begin{proposition}\label{amplif_1_1}
For any $m\ge 0$ the map $\Psi-d\Psi(0)$ defines an analytic mapping from $H^m$
 to $h^{m+1}$.
\end{proposition}

i) A local version of the last statement which deals with the germ of $\Psi$ at the
 origin, is established in \cite{KPer10}.

ii) Consider the restriction of $\Psi$ to the subspace $H^m_{\rm even}\subset H^m$ formed
 by even functions. The map $\Psi$, how it is defined in \cite{KaP}, maps
$H^m_{\rm even}$ to the subspace $h^m_e\subset h^m$, where
$$
h_e^m=\{v=({\bf v}_1,{\bf v}_2,\dots)\,:\,v_{-j}=0 \ \ \hbox{for all }j\in\mathbb Z^+ \}.
$$
For $u\in H^m_{\rm even}$ we have $u_{-s}=0$ for any $s\in\mathbb Z^+$.
 Thus $H^m_{\rm even}$ may be identified with the space
$$
l^2_m=\big\{(u_1,u_2,\dots)\,:\, \sum\limits_{r=1}^\infty r^{2m}u_r^2<\infty\big\}.
$$
Considering the asymptotic expansion for the actions $I_j$ for $u\in H^m$ we have $\
v_j=\pm\sqrt{2I_j}=j^{-1/2}(u_j+j^{-1}w_j(u)),
$
where the map $u\mapsto w(u)$ from $l^2_m$ into itself, $m\ge 0$, is analytic.
See for instance Theorem 1.2, formula (1.13), in \cite{Kor08} where one should
 use the Marchenko-Ostrovskii asymptotic formula to relate Fourier coefficients
 of a potential with the sizes $\gamma_n$ of open gaps of the corresponding Hill
 operator. Thus, for the restriction of $\Psi$ to the space $H_{\rm even}^m$ the
 assertion also holds.

 iii) The $n$-gap manifold $\TT^{2n}$ is the set of all $u(x)$ such that
 $v=\Psi(u)$ satisfies  $\vv_j=0$ if $j\ge n+1$. This is a  $2n$-dimensional analytic
 submanifold of any space $H^m$. It passes through $0\in H^m$ and goes to
 infinity; it can be defined independently from the map $\Psi$, see \cite{K2, KaP}. In
 a suitable neighbourhood of $\TT^{2n}$ there is an analytical transformation which
 put the KdV equation to a partial Birkhoff normal form (sufficient for
 purposes of the KAM-theory). The non-linear part of this map also is one-smoother
 than  its linear part, see in \cite{K2}.
 \smallskip

Proof of the Proposition in the general case, based on the spectral theory of
Hill operators, will be given in a separate publication.

\subsection{Equation (\ref{kdv}) in Birkhoff coordinates}

Applying the It\^o formula to the nonlinear Fourier transform  $\Psi$, we see that
for $u(t)$, satisfying  (\ref{kdv}), the function ${v}(\tau)=\Psi(u(\tau))$,
where $\tau=\nu t$, is a  solution of  the system
\begin{equation}\label{kdv_bir}
d\vv_k=\nu^{-1}d\Psi_k(u)V(u)d\tau+P_k^1({v})d\tau+P_k^2({v})d\tau
+\sum\limits_{j\ge1}B_{kj}({v})d\bb_j(\tau),\quad k\ge1.
\end{equation}
Here $\bb_j=\left(
\begin{array}{c}
\beta_j\\
\beta_{-j}
\end{array}
\right)\in \mathbb R^2$, $V(u)=-u_{xxx}+6uu_x$ is the vector field of  KdV,
 $P^1({v})=d\Psi(u)u_{xx}$ and  $P^2({v})d\tau$ is the It\^o term,
$$
P_k^2({v})=\frac{1}{2}\sum\limits_{j\ge1}b_j^2\left[d^2
\Psi_{kj}(u)\left(\Big({1\atop 0}\Big),\Big({1\atop 0}\Big)\right)+ d^2
\Psi_{kj}(u)\left(\Big({0\atop 1}\Big),\Big({0\atop 1}\Big)\right)
\right]\in\R^2.
$$
Finally, the dispersion matrix $B$ is formed by $2\times2$-blocks $B_{kj},\  k,j\ge1$,
 where
$$
B_{kj}(u)=b_j\left(d\Psi(u)\right)_{kj}.
$$
Equation (\ref{kdv_bir}) implies the following relation for the actions vector
$I=(I_1,I_2,\dots)$:
\begin{equation}\label{eq_for_i}
dI_k=\vv_k^tP_k^1({v})d\tau+\vv_k^tP_k^2({v})d\tau+ \frac{1}{2}
\sum\limits_{j\ge1}\|B_{kj}\|^2_{HS}d\tau+ \sum\limits_{j\ge1}\vv_k^tB_{kj}({v})
d\bb_j(\tau)\,,
\end{equation}
$k\ge1$.
Here  $\|B_{kj}\|^2_{HS}$ is the squared Hilbert-Schmidt norm of the $2\times2$
matrix $B_{kj}$, i.e. the sum of squares of all its four elements.

Estimates \eqref{a_est1} and eq.~\eqref{eq_for_i}  imply that
\begin{equation}\label{xxx}
    \E\sup_{0\le\tau\le T}|I(\tau)|^k_{h^m_I}\le C_{m,k}\qquad\forall\, m,k\ge0.
\end{equation}
See in \cite{KP08}.

\subsection{Averaged equations}


For a vector ${v}=(\vv_1,\vv_2,\dots)$ denote by $\varphi(v)=(\varphi_1,
 \varphi_2,\dots)$
 the vector of angles. That is $\varphi_j$ is the argument of the vector $\vv_j\in\R^2$,
 $\varphi_j=\arctan(v_{-j}/v_j)$ (if $\vv_j=0$, we set $\varphi_j=0$). The vector
 $\vp(v)$ belongs to the infinite-dimensional torus $\T^\infty$. We provide the latter
with the Tikhonov topology (so it becomes a compact metric space) and the Haar measure
$d\theta=\prod (d\theta_j/2\pi)$.
 We will identify
 a vector $v$ with the pair $(I,\varphi)$ and  write $v=(I,\vp)$.

The torus $\mathbb T^\infty$ acts on each
 space $h^m$ by the linear rotations $\Phi_\theta,\theta\in\T^\infty$,
 where $\Phi_\theta\,:\,(I,\varphi)
\mapsto(I,\varphi+\theta)$.
For any continuous function $f$ on $h^m$ we denote by $\langle f\rangle$ its angular
 average,
$$
\langle f\rangle({v})=\int\limits_{\mathbb T^\infty}f(\Phi_\theta{v})d\theta.
$$
 The function
$\langle f\rangle({v})$ is as smooth as $f({v})$ and depends only on $I$.
Furthermore, if $f({v})$ is analytic on $h^m$, then  $\langle f\rangle({I})$
is analytic on $h^m_I$; for the proof see \cite{KP08}.

Averaging equations  (\ref{eq_for_i}) using the rules of stochastic averaging
 (see \cite{Khas68, FW03}), we get the averaged system
\begin{equation}\label{aveq_i}
\begin{split}
\displaystyle
dI_k(\tau)&=\langle \vv_k^tP_k^1\rangle(I)d\tau+ \langle \vv_k^tP_k^2\rangle(I)d\tau \\
&+\frac{1}{2}\left\langle \sum\limits_{j\ge1} \|B_{kj}\|_{HS}^2\right\rangle(I)\,
d\tau+\sum\limits_{j\ge1} K_{kj}(I)d\bb_j(\tau),\qquad k\ge 1\,,
\end{split}
\end{equation}
with the initial condition
\begin{equation}\label{IC}
    I(0)=I_0=I(\Psi(u_0)).
\end{equation}
Here the dispersion matrix $K$ is a square root of the averaged diffusion matrix $S$,
\begin{equation}\label{diff}
S_{km}(I)\mathop{=}\limits^{\rm def}\left\langle\sum\limits_{l\ge1}\vv_k^tB_{kl}
 \vv_m^tB_{ml}\right\rangle(I),
\end{equation}
not necessary symmetric. That is,
\begin{equation}\label{sqrt}
\sum\limits_{l\ge1}K_{kl}(I)K_{ml}(I)=S_{km}(I)
\end{equation}
(we abuse the language since the l.h.s. is not $K^2$ but $KK^t$). If in \eqref{aveq_i}
we replace $K$
by another square root of $S$, we will get a new equation which has the same set
of weak solutions, see \cite{Yor74}.

Note that system  (\ref{aveq_i}) is very irregular: the operator
 $\langle G_k^1\rangle$ is unbounded and the matrix $K(I)$ is not
 Lipschitz continuous in $I$.

\subsection{Averaging principle}

Let us fix any $p\ge3$ and denote
\begin{equation}\label{notat}
    \cH_I=C([0,T],h^p_{I+}),\qquad \cH_v=C([0,T],h^p).
\end{equation}
In \cite{KP08} we have proved the following results:
given any  $T>0$, for the process
 $I^\nu(\tau)=\{I(v^\nu(\tau))\,:\,0\le\tau\le T\}$ it holds
\begin{theorem}\label{t_comp1}
Let $u^\nu(t)$, $0<\nu\le 1$, be a solution of (\ref{kdv}), (\ref{k0}) and
$v^\nu(\tau)=\Psi(u^\nu(\tau))$, $\tau=\nu t$, $\tau\in[0,T]$.  Then the family of measures
${\cal D}(I^\nu(\cdot))$ is tight in the space of (Borel) measures in $\cH_I$.
 Any limit point of this family, as $\nu\to0$, is the distribution of a weak solution
$I^0(\tau)$ of the averaged equation (\ref{aveq_i}), (\ref{IC}).
It satisfies the estimates
\begin{equation}\label{est}
\mathbb E\,\sup\limits_{0\le \tau\le T}|I^0(\tau)|^N_{h_I^m}<\infty
\quad \forall\,m,N\in\N,
\end{equation}
and
\begin{equation}\label{e120}
\E\int_{0}^{T}\chi_{\{I^0_k(\tau)\le\delta\}}(\tau)\,d\tau\to0\quad \text{as }\;\
\delta\to0,
\end{equation}
for each $k$.
\end{theorem}

\noindent{\it Remarks.} 1) The convergence \eqref{e120} is proved
in Lemma~4.3 of \cite{KP08}.
There is a flaw in the {\it statement}
of Lemma~4.3: the convergence \eqref{e120} is there claimed for any fixed $\tau$
 (without integrating in $d\tau$). This is true only for the case of stationary
solutions, cf. the next remark.
The proof of the main results in \cite{KP08} uses exactly \eqref{e120},
 cf. there estimate (5.7). See below Appendix, where the proof of Lemma~4.3 is re-written
 for purposes of this work.

2)
A similar result holds when $ u^\nu(t)= u_{\rm st}^\nu(t)$, $t\ge 0$,
 is a stationary solution of (\ref{kdv}), see \cite{KP08}.
\medskip

\subsection{Dispersion matrix $K$}

The matrix $S(I)$ is symmetric and positive but its spectrum contains $0$.
Consequently, its symmetric square root $\sqrt{S}(I)$ has low regularity in $I$
\footnote{Matrix elements of $\sqrt{S}(I)$ are Lipschitz functions of the arguments
$\sqrt{I_1}, \sqrt{I_2},\dots$. Cf.  \cite{IW}, Proposition~IV.6.2.}
at points of the set
 $$
 \p\, h^p_{I+}=\{I\in h^p_{I+} :I_j=0\quad\text{for some}\quad j\}.
 $$
Now we construct a `regular' square root $K$ (i.e. a dispersion matrix) which is
 an analytic function of $v$, where $I(v)=I$. This regularity will be sufficient for our
  purposes.

We will obtain a dispersion matrix $K=\{K_{lm}\}(v)$, $I(v)=I$,  as the matrix of a
 dispersion operator
$\
{\bf K}\,:\,Z\ \longrightarrow\ l_2,
$
where $Z$ is an auxiliary separable Hilbert space and the operator depends on the
parameter $v$, ${\bf K}={\bf K}(v)$.
The matrix $K$ is written with respect to some orthonormal basis in $Z$ and the
 standard basis $\{f_j, j\ge1\}$  of $l_2$. Below for a space $Z$ we take a
 suitable $L^2$-space $Z=L^2(X,\mu(dx))$. For any Schwartz kernel $\MM(v)=\MM(j,x)(v)$,
depending on the parameter $v$, we denote by $\Op(\MM(v))$ the corresponding integral
 operator from $L^2(X)$ to $l_2$:
$$
\Op(\MM(v))\,g(\cdot)=\sum_j f_j\int \MM(j,x)(v)g(x)\,\mu(dx).
$$
We will define the dispersion operator
 ${\bf K}(v)$ by its Schwartz kernel ${\KK}(j,x)(v)$, ${\bf K}(v)=\Op (\KK(v))$.
  For any choice of the
 orthonormal basis in $Z$ the Percival identity holds:
\begin{equation}\label{Parc}
\sum\limits_{l\ge1} K_{kl}(v)K_{ml}(v)\,=\,\int_X \KK(k,x)(v)\KK(m,x)(v)\,\mu(dx)\quad
\forall\,k,m.
\end{equation}

Since a law of a zero-meanvalue Gaussian process is defined by its correlations, then
due to \eqref{Parc} the law of the process
$\
\sum_{l\ge1} f_l
\sum_{m\ge1} K_{l m}\beta_m(\tau) \in l_2
$
does not depend on the choice of the orthonormal basis in $Z$: it depends only
on the correlation  operator $\bf K$ (i.e. on its kernel
 $\KK$) and not on a matrix $K$. Accordingly, we will {\it formally}
 denote the differential of this process as
\begin{equation}\label{symb}
\sum_{l\ge1}
f_l\sum_{m\ge1}K_{lm}\,d\beta_m(\tau)=\sum_{l\ge1}
f_l
\int_X \KK(l,x)d\beta_x(\tau)\,\mu(dx),
\end{equation}
where $\beta_x(\tau),\ x\in X$, are standard independent Wiener processes on some
probability space.
\footnote{We cannot find continuum independent copies of a random variable on a
standard probability space. So indeed this  is just a notation.}
Naturally, if in a stochastic equation the diffusion is written in the
 form \eqref{symb}, then
only weak solutions of the equation are well defined.
 This notation well agrees with the It\^o formula. Indeed, denote the
 differential in \eqref{symb} by $d\eta$ and let $f(\eta)$ be a $C^2$-smooth
 function. Then due to \eqref{Parc}
\begin{equation}\label{Ito}
\begin{split}
&df(\eta)=\Big(\frac12 \sum_{k, r}\frac{\p^2 f}{\p\eta_k \p\eta_r}\sum_m
 K_{km}K_{rm}\Big)d\tau    +
\sum_{k,m}\frac{\p f}{\p\eta_k}K_{km}\,d\beta_m(\tau)\\
&=\big(\frac12 \sum_{k,r}\frac{\p^2 f}{\p\eta_k \p\eta_r}\int_X
\KK(k,x)\KK(r,x)\,\mu(dx)\Big)d\tau
\\
&\qquad\qquad\qquad
+\sum_{k}\frac{\p f}{\p\eta_k}\int_X \KK(k,x)\,d\beta_x(\tau)\,\mu(dx).
\end{split}
\end{equation}
\smallskip

Due to \eqref{Parc} the matrix $K(v)$ satisfies equation \eqref{sqrt} if
\begin{equation}\label{kernrel}
\begin{split}
\int_X \KK(k,x)(v)\KK(m,x)(v)\mu(dx)=\sum\limits_{l\ge1} K_{kl}(v)K_{ml}(v)\\
=S_{km}(I)=\sum\limits_{l\ge1} \left\langle (\vv_k^tB_{kl}(v)) (\vv_m^tB_{ml}(v))
 \right\rangle.
\end {split}
\end{equation}
The matrix in the right-hand side of (\ref{kernrel}) equals
\begin{equation*}
    \begin{split}
&\sum\limits_{l\ge1}\,\int\limits_{\mathbb T^\infty}\left(
(\vv_k^tB_{kl})(\Phi_\theta v)
\right)
\left(
 (\vv_m^tB_{ml})(\Phi_\theta v)\right)d\theta\\
&
=\vv_k^t\vv_m^t \sum\limits_{l\ge1}\,\int\limits_{\mathbb T^\infty}
(\Phi^k_{-\theta_k}B_{kl}(\Phi_\theta v)) (\Phi^m_{-\theta_m}B_{ml}(\Phi_\theta v))d\theta\,,
\end{split}
\end{equation*}
where $\Phi^m_\theta$ is the linear operator in $\R^2$, rotating the $\vv_m$-component
 of a vector ${v}$ by the angle $\theta$.
Let us choose for  $X$ the space $X=\mathbb Z^+\times\mathbb T^\infty=\{(l,\theta)\}$
 and equip it with the measure $\mu(dx)=dl \times d\theta$, where $dl$ is
 the counting measure  in $\mathbb Z^+$ and $d\theta$ is the Haar measure in $\T^\infty$.
 Consider the following  Schwartz kernel $\KK$:
\begin{equation}\label{def_r}
\KK(k;l,\theta)(v)=\vv_k^t\RR(k;l,\theta)(v),\quad
\mathcal{R}(k;l,\theta)({v})=(\Phi_{-\theta_k}^kB_{kl})(\Phi_\theta({v})).
\end{equation}
Then \eqref{kernrel} is fulfilled. So
\begin{equation}\label{ssqrt}
\begin{split}
&\text{for any choice of the basis in
  $L_2(\mathbb Z^+\times\mathbb T^\infty)$
  } \\
  &\text{the matrix $K(v)$ of $\Op(\KK(v))$ satisfies  $\eqref{sqrt}$ with $I=I(v)$}.
 \end{split}
 \end{equation}
The differential \eqref{symb} depends on $v$, but its law depends only on $I(v)$.

We  formally write the averaged equation \eqref{aveq_i} with the constructed above
dispersion operator $\Op(\KK(v))$,  $I(v)=I$, as
\begin{equation}\label{0.0}
    \begin{split}
dI_k(\tau)=\langle \vv_k^tP_k^1\rangle(I)\,d\tau&+\langle
\vv_k^tP_k^2\rangle(I)\,d\tau +
\frac12\left\langle\sum_{j\ge1}\|B_{kj}\|^2_{HS}
\right\rangle(I)\,d\tau \\
&+\sum_{l\ge1} \int_{\T^\infty}\vv_k^t \RR(k,l,\theta)(v)\,d\bb_{l,\theta}(\tau)
\,d\theta.
    \end{split}
\end{equation}

Let us fix a basis in the space $L_2(\mathbb Z^+\times\mathbb T^\infty)$
 and fix the Wiener processes $\{\beta_m(\cdot), m\ge1\}$, corresponding
 to the presentation
 \eqref{symb} for the
stochastic term in \eqref{0.0}. Let $\xi\in h^p$ be a random variable,
 independent from the processes $\{\beta_m(\tau)\}$.

\begin{definition} I)
A pair of processes $I(\tau)\in h^p_I, v(\tau)\in h^p$, $0\le\tau\le T$,
 such that $I(v(\tau))\equiv I(\tau)$, $v(0)=\xi$ and
\begin{equation}\label{bound}
    \E \sup_{0\le\tau\le T}|v(\tau)|_m^N<\infty \quad \forall\,m,N,
\end{equation}
is called a regular strong  solution of \eqref{0.0} in the space $h^p_I\times h^p$,
 corresponding to the basis above and
 the Wiener processes $\{\beta_m(\cdot)\}$, if

(i) $I$ and $v$ are adapted to the filtration, generated by $\xi$ and the
 processes $\{\beta_m(\tau)\}$,


(ii) the integrated in $\tau$  version of \eqref{0.0} holds a.s.
\smallskip

II) A pair of processes $(I,v)$ is called a  regular weak  solution if it
 is a regular solution for some choice of the basis and
 the Wiener processes $\{\beta_m\}$, defined on a suitable extension of the original
 probability space (see in \cite{KaSh}).
\end{definition}

\begin{lemma}\label{l_lift}
If $(I(\tau), v(\tau)), 0\le\tau\le T$, is a  regular weak    solution of
 eq.~\eqref{0.0}, then $I(\tau)$ is a weak solution of \eqref{aveq_i},
 where $K_{km}(I)$ is the symmetric square root $\sqrt{S_{km}(I)}$.
\end{lemma}

{\it Proof.} Clearly the process $I(\tau)$ is a solution to the (local)
 martingale problem, associated with eq.~\eqref{aveq_i} (see \cite{KaSh},
 Proposition~4.2 and Problem~4.3). So $I(\tau)$ is a weak solution of
 \eqref{aveq_i}, see \cite{Yor74}
and Corollary~6.5  in \cite{KP08}.
\qed
\medskip

The representation of the averaged equation \eqref{aveq_i} in the form \eqref{0.0} is
crucial for this work. It is related to the construction of non-selfadjoint dispersion
operators in the work \cite{PP} and is inspired by it. We are thankful to A.~Piatnitski
for corresponding discussion.

\section{Effective  equations}\label{s_lifteq}

The goal of this section is to lift the averaged equation (\ref{aveq_i})
 to an  equation for the vector $v(\tau)$ which transforms to
 \eqref{aveq_i} under the mapping $v\mapsto I(v)$.
Using Lemma~\ref{l_lift} we instead lift equation \eqref{0.0}. We
 start the lifting  with the last two terms in the right hand side
 of (\ref{0.0}). They define the It\^o differential
\begin{equation}\label{zadn_ch}
\frac{1}{2}\left\langle\sum\limits_{j\ge1}\|B_{kj}\|^2_{\rm HS}\right\rangle(I)\,
 d\tau+
\sum\limits_{l\ge1}\,\int\limits_{\mathbb T^\infty}\vv_k^t \mathcal{R}(k;l,
\theta)({v}) d\bb_{l,\theta}(\tau)d\theta.
\end{equation}
Consider the  differential
$\
d\vv_k=\sum\limits_{l\ge1}\,\int\limits_{\mathbb T^\infty}\mathcal{R}(k;l,\theta)({v})
 d\bb_{l,\theta}(\tau)d\theta.
$
Due to \eqref{Ito}, for  $J_k=\frac12|\vv_k|^2$ we have
\begin{equation*}
dJ_k=\frac{1}{2}\bigg(\sum\limits_{l\ge1}\,\int\limits_{\mathbb T^\infty}\|
\mathcal{R}(k;l,\theta)\|^2_{\rm HS}d\theta\bigg)d\tau+ \sum\limits_{l\ge1}\,
\int\limits_{\mathbb T^\infty}\vv_k^t \mathcal{R}(k;l,\theta)({v})
 d\bb_{l,\theta}(\tau)d\theta.
\end{equation*}
Notice that the diffusion term in the last formula coincides with that in
 (\ref{zadn_ch}). The drift terms also are the same since
$\
\|\Phi^k_{\theta'}B_{kl}\|_{\rm HS}^2= \|B_{kl}\|_{\rm HS}^2
$
for any rotation $\Phi^k_{\theta'}$.

Now consider the first part of the differential in the right-hand side of (\ref{aveq_i}),
\begin{equation}\label{j3_def }
\left\langle\vv_k^tP_k^1\right\rangle(I)d\tau+
 \left\langle\vv_k^tP_k^2\right\rangle(I)d\tau.
\end{equation}
Recall that $P^1=d\Psi(u)u_{xx}$ with $u=\Psi^{-1}({v})$ and that $P^2({v})$
 is the It\^o term.  We have
\begin{equation*}
\begin{split}
\left\langle\vv_k^tP_k^1\right\rangle(I)&=
\int\limits_{\mathbb T^\infty}
(\vv_k^tP_k^1)(\Phi_\theta{v})d\theta=\int\limits_{\mathbb T^\infty}\vv_k^t
\left(\Phi_{-\theta_k}^k\,d\Psi_k(\Pi_\theta u)
\frac{\partial^2}{\partial x^2}\big(\Pi_\theta u\big)\right)d\theta\\
&=
\vv_k^tR_k^1(v),\qquad \;\;  u=\Psi^{-1}(v),
\end{split}
\end{equation*}
where
$\
R_k^1(v)=\int\limits_{\mathbb T^\infty}\Phi_{-\theta_k}^k\,d\Psi_k(\Pi_\theta u)
\frac{\partial^2}{\partial x^2}\big(\Pi_\theta u\big)d\theta$,
  and the operators $\Pi_\theta$ are defined by the relation
$\
\Pi_\theta u=\Psi^{-1}(\Phi_\theta v).
$
Similarly,
$$
\left\langle\vv_k^tP_k^2\right\rangle(I)=\int\limits_{\mathbb T^\infty}
(\vv_k^tP_k^2)(\Phi_\theta{v})d\theta=
\vv_k^t\int\limits_{\mathbb T^\infty}\Phi^k_{-\theta_k}P^2_k(\Phi_\theta v)\,d\theta=:
\vv_k^tR_k^2(v).
$$
Consider the differential $d\vv_k=R^1_k(v)\,d\tau + R^2_k(v)\,d\tau$.
Then $d\left(\tfrac12|\vv_k|^2\right)=\,$\eqref{j3_def }.

Now consider the system of  equations:
\begin{equation}\label{lif_sy}
d\vv_k(\tau)=R^1_k({v})d\tau+R^2_k({v})d\tau+\sum_{l\ge1}\,
\int\limits_{\mathbb T^\infty}\mathcal{R}(k;l,\theta)({v})
d\bb_{l,\theta}(\tau)d\theta,\quad k\ge 1.
\end{equation}

The arguments above prove that if $v(\tau)$ satisfies \eqref{0.0},
 then $I(v(\tau))$ satisfies (\ref{aveq_i}).
Using Lemma~\ref{l_lift} we get
\begin{proposition}\label{t_lif00}
If ${v}(\tau)$ is a regular weak solution of equation (\ref{lif_sy}),
then\\  $I(v(\tau))$ is a regular weak solution
 of \eqref{aveq_i}.
\end{proposition}

Here a {\it regular weak solution} is a weak solution, satisfying \eqref{bound}.
\smallskip

The drift $R^1(v)+R^2(v)$ in the effective equations \eqref{lif_sy} is an
averaging of the vector-field $P(v)=P^1(v)+P^2(v)$, see \eqref{eff}.

The kernel $\RR(k;l,\theta)(v)$ defines a linear operator $\RRR(v):=\Op(\RR(v))$
from the space  $L_2:=L_2(\Z^+\times\T^\infty)$ to the space $h:=h^{-1/2}$ (so the space $h$ is given the $l_2$-scalar product), see Section~1.5. The operator
 $\RRR(v)\RRR(v)^t:h\to h$ has the matrix $X(v)$, formed by $2\times2$-blocks
 $$
 X_{kj}(v)=\sum_l\int_{\T^\infty}\RR(k;l,\theta)(v)\,\RR(j;l,\theta)(v)\,d\theta.
 $$
 Due to \eqref{def_r} this is the matrix of the averaged diffusion operator \eqref{AvDiff}. If we write the diffusion term in the effective equations in the
 standard form, i.e. as
 $\sum_j\lan\lan B\ran\ran_{kj}(v)\,d\beta_j(\tau)$, where $\lan\lan B\ran\ran$ is
 a matrix of the operator $\RRR(v)$ with respect to some basis in $L_2$
 (see \eqref{symb}), then also $\lan\lan B\ran\ran (v)\lan\lan B\ran\ran^t(v)=X(v)$,
 see \eqref{Parc}. So the dispersion operator in \eqref{lif_sy} is a non-symmetric
 square root of the averaged diffusion operator in the $v$-equations. Cf. relation \eqref{effD} and its discussion. 
 \medskip

System \eqref{0.0} has locally Lipschitz coefficients
and does not have a singularity at $\p\,h^I_{p+}$, but its dispersion operator
 depends on $v$.
 Now we construct an equivalent
system of equations on $I$ which is $v$-independent, but has  weak singularities
 at  $\p\,h^I_{p+}$.

 The dispersion kernel in equation \eqref{0.0} is
$\vv_k^t\RR(k;l,\theta)(v)$. Let us re-denote it as $\KK_k(l,\theta)(v)$. Then
$\
\KK_k(l,\theta)(v)=\vv_k^tB_{kl}(v) \mid_{v:=\Phi_\theta v}
$.
Clearly
\begin{equation}\label{e00}
\KK_k(l,\theta)(\Phi_\phi v)= \KK_k(l,\theta+\phi)(v).
\end{equation}
Denoting, as before,  by $\Op(\KK(v))$ the linear
 operator $L_2(\N\times\T^\infty)\to l_2$ with the
kernel  $\KK(v)=\KK_k(l,\theta)(v)$, $v=(I,\vp)$, we have
\begin{equation}\label{e0}
\Op\big(\KK(I,\vp_1+\vp_2)\big)=\Op\big(\KK(I,\vp_1)\big)\circ U{(\vp_2)}.
\end{equation}
Here $U(\vp)$ is the unitary operator in $L_2(\N\times\T^\infty)$, corresponding to the
 rotation of $\T^\infty$ by an angle $\vp$.

Let us provide $L_2(\T^1, dx/2\pi)$ with the basis
 $\xi_j(\theta), \ j\in\Z$, where $\xi_0=1$,
$\xi_j=\sqrt2\,\cos jx$ if $j\ge1$ and $\xi_j=\sqrt2\,\sin jx$
 if $j\le-1$. For $i\in\Z$ and $s=(s_1,s_2,\dots)\in\Z^\N $, $|s|<\infty$, define
$$
E_{i,s}(l,\theta)=\delta_{l-i}\prod_{j\in\Z}\xi_{s_j}(\theta_j)
$$
(the infinite product is well defined since a.a. factors is 1).
These functions define a basis in $L_2(\N\times \T^\infty)$. Let
 $(E_r ,r\in\N)$, be the same functions, re-parameterised
by the natural parameter. For any $v=(I,\vp)$ the matrix $\KK(v)$ with the elements
$$
\KK_{k r}(v)=
\Big(\KK_k(l,\theta)(v), E_r(l,\theta)\Big)_{L_2}=
\int_{\Z^+\times\T^\infty}\KK_k(l,\theta)(v)\, E_r(l,\theta)(dl\times d\theta)
$$
is the matrix of the operator  $\Op(\KK(v))$ with
 respect to the basis $\{E_r\}$.

Due to \eqref{e0} for $v=(I,\vp)$ the operator
$\Op(\KK(I,\vp))$ equals $\Op(\KK(I,0))\circ U(\vp)$. So its matrix is
$$
\KK_{kr}(I,\vp)=\sum_mM_{k m}(I)U_{mr}(\vp),
$$
where the matrix $M_{km}(I)$  corresponds to the kernel
 $\KK_k(l,\theta)(I,0)$  and $U_{mr}(\vp)$ is the matrix of the operator $U(\vp)$
 (the matrices are formed by $2\times2$-blocks).
 Clearly $\|K(I,\vp)\|_{HS}=\|M(I)\|_{HS}$ for each $(I,\vp)$.
 Taking into account the form of the functions $E_{i,s}(l,\theta)$ we see that
 any  $U_{mr}(\vp)$  is a smooth function of each argument
 $\vp_j$ and is  independent from $\vp_k$ with $k$ large enough. In particular,
\begin{equation}\label{l.e1}
\begin{split}
 \text{any matrix element $U_{mr}(\vp)$  is a Lipschitz function of
 $\vp\in\T^\infty$.}
 \end{split}
 \end{equation}
 Note that the Lipschitz constant of $U_{mr}$ depends on $m$ and $r$.

Let us  denote the drift in the system \eqref{0.0} by $F_k(I)\,d\tau$  and write the
dispersion matrix with respect to the basis $\{E_r\}$. It  becomes
\begin{equation}\label{e2}
dI_k(\tau)=F_k(I)\,d\tau + \sum_{m,r}M_{km}(I)\,U_{mr}(\vp)\,d\beta_r.
\end{equation}
Let $\vp(\tau)\in\T^\infty$ be any progressively measurable process with
 continuous trajectories. Consider the processes $\tilde\beta_m(\tau)$, $m\ge1$,
\begin{equation}\label{relat}
d\tilde\beta_m(\tau)=\sum_r  U_{mr}(\vp(\tau))\,d\beta_r(\tau),\qquad
 \tilde\beta_m(0)=0.
\end{equation}
Since $U$ is an unitary operator, then $\tilde\beta_m(\tau), \ m\ge1$, are standard
 independent Wiener processes. So we may write \eqref{e2} as
\begin{equation}\label{e3}
    dI_k(\tau)=F_k(I)\,d\tau+\sum_m M_{km}(I)\,d\tilde\beta_m(\tau).
\end{equation}
Note that each weak solution of \eqref{e3} is a weak solution of \eqref{e2}
 and vice versa. Due to \eqref{ssqrt} the matrix $M$ satisfies \eqref{sqrt}.
  So equation \eqref{e3} has the same  weak solutions as equation \eqref{aveq_i}.

Now consider system
 \eqref{lif_sy} for $v(\tau)$. Denote by $\RR_{km}(v)$ the matrix, corresponding
 to the  kernel $\RR(k;l,\theta)(v)$ in the basis $\{E_k\}$. Denoting $R^1_k+R^2_k=R_k$
we write \eqref{lif_sy} as follows:
\begin{equation}\label{e44}
d\vv_k=R_k(v)\,d\tau+\sum_r\RR_{kr}(v)\,d\beta_r(\tau)\qquad\qquad
\end{equation}
$$
\qquad\qquad
=R_k(v)\,d\tau+\sum_{m,l,r}\RR_{kl}(v)U_{ml}(\vp)U_{mr}(\vp)\,d\beta_r(\tau).
$$
So
\begin{equation}\label{e4}
    d\vv_k=R_k(v)\,d\tau+\sum_m\tilde\RR_{km}(v)\,d\tilde\beta_m(\tau),
    \qquad k\ge1,
\end{equation}
where
$\
\tilde\RR_{km}(v)=\sum_l\RR_{kl}(v)U_{ml}(\vp).
$
As before, equations \eqref{lif_sy} and \eqref{e4} have the same sets of
 weak solutions.
 Since matrix elements $U_{mr}(\vp)$ smoothly depend on $\vp$,  we have
\begin{equation}\label{e5}
\begin{split}
&\|\tilde\RR(v)\|_{HS}=\|\RR(v)\|_{HS}<\infty\;\forall\, v \;
\\
 &\text{and every $\tilde\RR_{kl}(v)\ $}\text{\ smoothly depends on each
 $v_r\in\R^2\setminus\{0\}$.
}
\end{split}
\end{equation}

We have established
\begin{lemma}\label{l.equiv}
Equations \eqref{e4} have the same set of regular weak solutions as
 equations \eqref{e44}, and
equations \eqref{e3} -- as equations \eqref{aveq_i}.
 The Wiener processes $\{\beta_r(\tau),\, r\ge1\}$
and $\{\tilde\beta_m(\tau)$, $m\ge1\}$ are related by formula \eqref{relat}, where
 $v(\tau)=(I(\tau), \vp(\tau))$ and  the unitary matrix $U(\vp)$ satisfies \eqref{l.e1}.
\end{lemma}

We also note that if a process $v(\tau)$ satisfies only one equation
\eqref{e4}, then it also
satisfies the corresponding equation \eqref{e44}.

\section{Lifting of solutions}\label{s3}
\subsection{The theorem}
In this section we prove an assertion which in some sense is
 inverse to that of  Proposition~\ref{t_lif00}.
 For any $\vt\in\T^\infty$ and any vector $I\in h_I^p$ we set
\begin{equation*}
\begin{split}
V_\vt(I)=(\Vv_{\vt \,1},\Vv_{\vt\,2},\dots)\in h^p,\quad &\Vv_{\vt\,_r}=
\Vv_{\vt_r}(I_r),\quad \text{where}\\
&\Vv_\alpha(J)=
(\sqrt{2J}\cos\alpha,\sqrt{2J}\sin\alpha)^t\in \R^2\,.
\end{split}
\end{equation*}
Then $\vp_j(V_\vt(I))=\vt_j\  \forall\,j$ and for every $\vt$ the map
$I\mapsto V_\vt(I)$ is right-inverse to the map $v\mapsto I(v)$.
 For $N\ge1$ and any vector $I$ we  denote
$$
I^{>N}=(I_{N+1},I_{N+2},\dots),\qquad
V^{>N}_\vt(I)=(\Vv_{\vt\,N+1}(I),\Vv_{\vt\,N+2}(I),\dots).
$$

\begin{theorem}[Lifting]\label{l_lif}
Let $I^0(\tau)=(I^0_k(\tau)$, $k\ge1$, $0\le \tau\le T)$,  be a weak solution
of system (\ref{aveq_i}), constructed in Theorem~\ref{t_comp1}.
 Then, for any vector $\vt\in\T^\infty$,
 there is a regular weak solution $v(\tau)$ of system \eqref{lif_sy}  such that

i) the law of $I(v(\cdot))$ in the space $\cH_I$ (see \eqref{notat})
 coincides with that of $I^0(\cdot)$,

 ii)  $v(0)=V_{\vt}(I_0)$ a.s.
\end{theorem}

\begin{proof} {\bf Step 1.}  {\it Re-defining the equations for large amplitudes.}

For any $P\in\N$ consider the stopping time
$$
\tau_P=\inf\{\tau\in[0,T]\mid |v(\tau)|_p^2\equiv |I(v(\tau))|_{h^p_I}=P\}
$$
(here and in similar situations below $\tau_P=T$ if the set is empty).
For $\tau\ge\tau_P$ and each $\nu>0$ we re-define equations \eqref{kdv_bir} to the
trivial system
\begin{equation}\label{re_def}
    d\vv_k=
    b_k d\bb_k(\tau),\qquad k\ge1,
\end{equation}
and re-define accordingly equations \eqref{eq_for_i} and \eqref{aveq_i}. We will denote
the new equations as \eqref{kdv_bir}${}_P$, \eqref{eq_for_i}${}_P$ and
 \eqref{aveq_i}${}_P$. If $v_P^\nu(\tau)$ is a solution of \eqref{kdv_bir}${}_P$, then
$I_P^\nu(\tau)=I(v_P^\nu(\tau))$ satisfies \eqref{eq_for_i}${}_P$. That is, for
$\tau\le\tau_P$ it satisfies \eqref{eq_for_i}, while for $\tau\ge\tau_P$ it is
a solution of the It\^o equations
\begin{equation}\label{re_defI}
dI_k=\12b_k^2\,d\tau+b_k(v_k\,d\beta_k+v_{-k}\,d\beta_{-k})=
\12b_k^2\,d\tau+b_k\sqrt{2I_k}\,dw_k(\tau),\qquad k\ge1,
\end{equation}
where $w_k(\tau)$ is the Wiener process
$\int^\tau(\cos\vp_k\,d\beta_k+\sin\vp_k\,d\beta_{-k})$. So \eqref{eq_for_i}${}_P$
 is the system of equation
\begin{equation}\label{stop1}
dI_k=\chi_{\tau\le\tau_P}\cdot \lan {\rm r.h.s.\  of\  }\eqref{eq_for_i}\ran+
\chi_{\tau\ge\tau_P}\left(\12b_k^2\,d\tau +b_k\sqrt{2I_k}\,dw_k(\tau)\right),\quad k\ge1.
\end{equation}
Accordingly, the averaged system \eqref{aveq_i}${}_P$ may be written as
\begin{equation}\label{stop2}
dI_k=\chi_{\tau\le\tau_P}\Big( F_k(I)\,d\tau+\sum_jK_{kj}(I)\,d\beta_j(\tau)\Big)+
\chi_{\tau\ge\tau_P}\left(\12b_k^2\,d\tau +b_k\sqrt{2I_k}\,d\beta_k(\tau)\right),
\end{equation}
 $k\ge1$. Here (as in \eqref{e2}) $F_k\,d\tau$ abbreviates the drift in
 eq.~\eqref{aveq_i},  and for $\tau\ge\tau_P$ we replaced the Wiener process $w_k$ by
 the process $\beta_k$ -- this does not change  weak solutions  the system.

 Similar to $v^\nu$ and $I^\nu$ (see Lemma 4.1 in \cite{KP08}),
 the processes $v_P^\nu$ and $I_P^\nu$ meet the estimates
\begin{equation}\label{v_est}
    \E \sup_{0\le\tau\le T}|I(\tau)|^M_{h^m_I}=\E\sup_{0\le\tau\le T}|v(\tau)|_{h^m_I}^{2M}
    \le C(M,m,T),
\end{equation}
uniformly in $\nu\in(0,1]$.

Due to Theorem \ref{t_comp1} for a  sequence $\nu_j\to0$ we have
$
\DD(I^{\nu_j}(\cdot))\strela \DD(I^0(\cdot))
$.
Choosing a suitable subsequence we achieve that also
$
\DD(I_P^{\nu_j}(\cdot))\strela \DD(I_P(\cdot))
$
for some process $I_P(\tau)$, for each $P\in\N$. Clearly $I_P(\tau)$ satisfies
 estimates \eqref{v_est}.

\begin{lemma}\label{l.P}
For any $P\in\N$, $I_P(\tau)$ is a weak solution of \eqref{aveq_i}${}_P$ such
that $\DD(I_P)=\DD(I^0)$ for $\tau\le\tau_P$
 \footnote{That is, images of the two measures under the mapping
 $I(\tau)\mapsto I(\tau\wedge\tau_P)$ are equal. } and
$
\DD(I_P(\cdot))\strela \DD(I^0(\cdot))
$
as $P\to\infty$.
\end{lemma}
{\it Proof.}  The process $I_P^\nu(\tau)$ satisfies the system of It\^o
equations \eqref{eq_for_i}${}_P$=\eqref{stop1}  which we now abbreviate as
\begin{equation}\label{s1}
dI^\nu_{Pk}=\FF_k(\tau,v_P^\nu(\tau))\,d\tau+
\sum_j\Sig_{kj}(\tau,v^\nu_P(\tau))\,d\beta_j(\tau)\,,\quad k\ge1.
\end{equation}
Denote by $\lan \FF\ran_k(\tau,I)$ and $\lan\Sig\Sig^t\ran_{km}(\tau,I)$ the
 averaged drift and  diffusion. Then
$$
\lan \FF\ran_k=\chi_{\tau\le\tau_P}F_k(I)+\chi_{\tau\ge\tau_P}\12\,b_k,\quad
\lan\Sig\Sig^t\ran_{km}=
\chi_{\tau\le\tau_P}S_{km}(I)+\chi_{\tau\ge\tau_P}\delta_{km}b_k^22I_k
$$
(cf. \eqref{e2} and \eqref{diff}).
We claim that
\begin{equation}\label{s2}
\Upsilon^q_\nu:=\E\sup_{0\le\tau\le T}
\left| \int_0^\tau(\FF_k(s,v_P^\nu(s))-\lan\FF\ran_k(s,I^\nu_P(s))\,ds
\right|^q\to0\quad{\rm as} \quad\nu\to0,
\end{equation}
for $q=1$ and 4. Indeed, since $\FF_k=\lan\FF\ran_k$ for $\tau\ge\tau_P$ and
$v^\nu_P=v^\nu$, $I^\nu_P=I^\nu$ for  $\tau\le\tau_P$,  then
$$
\Upsilon^q_\nu\le \E\sup_{0\le\tau\le T}
\left| \int_0^\tau(\FF_k(s,v^\nu(s))-F_k(I^\nu(s))\,ds
\right|^q\,.
$$
But the r.h.s. goes to zero with $\nu$, see in \cite{KP08} Proposition~5.2 and relation
(6.17). So \eqref{s2} holds true.

Relations \eqref{s1} and \eqref{s2} with $q=1$ imply that for each $k$ the process
$\
Z_k(\tau)=I_k(\tau)-\int_0^\tau\lan\FF_k\ran\,ds
$,
regarded as the natural process on the space $\cH_I$, given the natural filtration
 and the measure $\DD(I_P)$,  is a square integrable martingale, cf. Proposition~6.3
 in \cite{KP08}. Using the same arguments and \eqref{s2} with $q=4$ we see that
 for any $k$ and $m$ the process $Z_k(\tau)Z_m(\tau)-\int_0^\tau\lan
 \Sig\Sig^t\ran_{km}\,ds$ also is a
$\DD(I_P)$-martingale. It means that the measure $\DD(I_P)$ is a solution
 of the martingale problem for eq.~\eqref{aveq_i}${}_P$=\eqref{stop2}. That
 is, $I_P(\tau)$ is a weak
solution of \eqref{aveq_i}${}_P$.
\smallskip

Since $\DD(I^\nu_P)=\DD(I^\nu)=:\IP^\nu$ for $\tau\le\tau_P$, then passing to the
limit as $\nu_j\to0$ we get the second assertion of the lemma. As
$\IP^\nu\{\tau_P< T\}\le CP^{-1}$ uniformly in $\nu$ (cf. \eqref{v_est}), then the last
assertion also follows.
\qed
\medskip

\noindent
{\bf Step 2.} {\it  Equation for $v^N$.}

By Lemma \ref{l.equiv} the process $I^0(\tau)$ satisfies \eqref{e3}.
For any $N\in\N$ we consider a
 Galerkin--like approximation for equations \eqref{e4}, coupled with eq.~\eqref{e3}.
 Namely, denote
$$
v^N(\tau)=(\vv_1,\dots,\vv_N)(\tau)\in\R^{2N},\quad
V^{>N}(\tau)=
V^{>N}_\vt(I(\tau))\,,
$$
and consider the following system
 of equations:
\begin{equation}\label{e7}
    \begin{split}
dI_k(\tau)=&F_k(I)\,d\tau+\sum_{m\ge1}M_{km}(I)\,d\tilde\beta_m(\tau),\quad k\ge1,\\
d\vv_k(\tau)=&R_k(v)\,d\tau+\sum_{m\ge1}\tilde\RR_{km}(v)
\,d\tilde\beta_m(\tau),
\quad k\le N,
    \end{split}
\end{equation}
where $v=(v^N, V^{>N}(I))$. We take $I(\tau)=I^0(\tau)$ for a solution of the
$I$-equations. Then  \eqref{e7} becomes equivalent
  to a system of $2N$ equations on $v^N(\tau)$
 with progressively measurable
 coefficients.

 As at Step~1 we re-define the $I$-equations in \eqref{e7} after $\tau_P$ to equations
 \eqref{re_defI} and the $v$-equations -- to \eqref{re_def}. We
  denote thus obtained system \eqref{e7}${}_P$. By Lemma~\ref{l.P} the
 process $I_P(\tau)$ satisfies the new $I$-equations, and we will take $I_P(\tau)$
 for the $I$-component of a solution for  \eqref{e7}${}_P$.
 To solve \eqref{e7}${}_P$ for $0\le\tau\le T$ we first solve
 \eqref{e7} till time $\tau_P$ and next solve the trivial system \eqref{re_def}
 for $\tau\in[\tau_P,T]$. The second step is obvious. So we will mostly  analyse
 the first step.
 The coefficients $R_k$ are Lipschitz in $v^N$ on bounded
subsets of $\R^{2n}$. Due to \eqref{e5} the  coefficients $\tilde\RR_{km}(v)$
are Lipschitz in $v^N$ if $|v|_p\le \sqrt{P}$ and
 $|\vv_j|\ >\delta$ $\forall\,j\le N$
 for some $\delta>0$, but the Lipschitz constants are not uniform in $m$.
 Denote
 $$
 \hat\Omega=\Omega_I\times\Omega_N=
 C(0,T;h^p_I)\times C(0,T;\R^{2N}),
 $$
 and denote by $\pi_I, \pi_N$ the natural projections $\pi_I:\hat\Omega\to\Omega_I$,
 $\pi_N:\hat\Omega\to\Omega_N$. Provide the Banach spaces $\hat\Omega, \Omega_I$
 and $\Omega_N$
 with the Borel sigma-algebras and the natural filtrations of sigma-algebras.

Our goal is  to construct a weak solution for \eqref{e7}${}_P$ such that its
 distribution
$\PP=\PP_P^N=\DD(I,v^N)$ satisfies
  $\pi_I\circ \PP=\DD(I_P(\cdot))$ and
 $I(v^N(\cdot))=I^N(\cdot)\ $ $\PP$-a.s.
After that we will go to a limit as $P\to\infty$ and $N\to\infty$ to get a required weak solution
$v$ of \eqref{lif_sy}.
\medskip

\noindent
{\bf Step 3.} {\it  Construction of a measure $\PP_\delta$.}

Let us denote
$
[I]=\min_{1\le j\le N}\{I_j\}
$.
Fix any positive $\delta$. For a process $I(\tau)$ we define
stopping times  $\theta^\pm_j\le T$ such that
 $\dots<\theta^-_j<\theta^+_j< \theta^-_{j+1}<\dots$  as follows:
\begin{itemize}
\item  if $[I(0)]\le\delta$, then $\thee=0$. Otherwise $\theta^+_0=0$.

\item If $\theta_j^-$ is defined, then $\theta_j^+$ is the first
  moment  after $\theta_j^-$ when $[I(\tau)]\ge 2\delta$ (if
  this never happens, then we set $\theta_j^+=T$; similar in the item below).

\item If $\theta_j^+$ is defined, then $\theta_{j+1}^-$ is the first
  moment  after $\theta_j^+$
when $[I(\tau)]\le \delta$.
\end{itemize}

We denote $\Delta_j=[\theta_j^-,\theta^+_j]$,
$\Lambda_j=[\theta_j^+,\theta^-_{j+1}]$  and set
$\Delta=\cup\Delta_j$, $\Lambda=\cup\Lambda_j$.

For segments $[0,\theta_j^-]$ and  $[0,\theta_j^+]$, which we denote below
 $[0,\theta_j^\pm]$, we will iteratively construct processes
 $(I,v^N)(\tau)=(I,v^N)^{j,\pm}(\tau)$ such that $\DD(I(\cdot))=\DD(I_P(\cdot))$,
  $v^N(\tau)=v^N(\tau\wedge\theta_j^\pm)$ and $\DD(I^N(\tau))=\DD(I(v^N(\tau))$
 for $\tau\le\theta_j^\pm$.  Moreover, on each segment $\Lambda_r\subset[0,\theta_j^\pm]$
the process  $(I,v^N)$  will be  a weak solution of \eqref{e7}${}_P$.
Next  we will obtain a desirable measure $\PP^N_P$ as a limit of the laws of
these processes  as  $j\to\infty$ and  $\delta\to0$.

For the sake of definiteness assume that $0=\theta_0^+$.
\smallskip

{\bf a)} $\tau\in\Lambda_0$. We will call the `$\delta$-stopped system  \eqref{e7}${}_P$' a system, obtained
 from  \eqref{e7}${}_P$ by multiplying the $v$-equations  by the  factor $\chi_{\tau\le\thee}$. We wish
 to construct a weak solution $(I,v^N)$ of this system such that, as before, $\DD(I)=\DD(I_P)$.
We will only show how to do this on the segment $[0, \thee\wedge\tau_P]$  since construction of
a solution for $\tau\ge\tau_P$ is trivial.

\begin{lemma}\label{l.e2}
For any positive $\delta$ and for $\vartheta$ as in Theorem~\ref{l_lif}
the  $\delta$-stopped system \eqref{e7}${}_P$ has a weak
 solution $(I,v^N)$  such that $\DD(I(\cdot))=\DD(I_P(\cdot))$ and
  $\12|\vv_k|^2(\tau)\equiv I_k(\tau)$, 
    $\vv_k(0)=\Vv_{\vt\,k}(I_0)$ for $k\le N$.
\end{lemma}
\begin{proof}
Let $(I,v^N)$ be a solution of \eqref{e7}${}_P$.
Application of the It\^o formula to $\varphi_k(v)=\arctan(v_k/v_{-k})$, $k\le N$,
 yields
\begin{equation}\label{e8}
d\varphi_k(\tau)=\chi_{\tau\le\thee}\left(
R^{atn}_k(v)\,d\tau+\sum\limits_{m\ge1}
\RR^{atn}_{km}(v)\,d\tilde\beta_{m}(\tau)\right),
\quad k\le N,
\end{equation}
where $v=(v^N,V^{>N})$ and
\begin{equation*}
\begin{split}
\RR^{atn}_{km} (v)=
\left(\nabla_{\vv_k}
\arctan\Big(\frac{v_k}{v_{-k}}\Big)\right)&\cdot\tilde{\cal R}_{km}(v),\\
R^{atn}_k(v)=
\left(\nabla_{\vv_k}
\arctan\Big(\frac{v_k}{v_{-k}}\Big)\right)&\cdot
R_k(v)\\
+\frac12\,&\sum\limits_{m\ge1}
\left(\nabla_{\vv_k}^2
\arctan\Big(\frac{v_k}{v_{-k}}\Big)\right)
\tilde{\cal R}_{km}\cdot\tilde{\cal R}_{km}.
\end{split}
\end{equation*}
Here $\cdot$ stands for the inner product in $\mathbb R^2$. In the r.h.s. of
\eqref{e8}, for $k=1,\dots,N$
we express $\vv_k(\tau)$ via $\vp_k(\tau)$ and $I_k(\tau)$ as
$\vv_k=\Vv_{\vp_k}(I_k)$.
Then
 $\chi_{\tau\le\thee} R_k^{atn}$ and $\chi_{\tau\le\thee}\RR_{km}^{atn}$
 become smooth functions
of $I$ and $\vp^M$. Accordingly,  \eqref{e3}${}_P$+ \eqref{e8}${}_P$ is a system
of equations for $(I,\vp^N)$ and the pair  $(I,\vp^N)$ as above is its solution.

Other way round, if a pair $(I,\vp^N)$ satisfies system
 \eqref{e3}${}_P$+ \eqref{e8}${}_P$,
then  $(I,v^N)$ is a solution of the $\delta$-stopped equations
\eqref{e7}${}_P$ such that
$\12|\vv_k|^2=I_k$ for $k\le N$.
Indeed, we recover the $v^N$-component of a solution $(I,v^N)$
as $\vv_j(\tau)=\Vv_{\vp_j(\tau)}(I_j(\tau))$, $j\le N$.

For any $M\ge1$ we call the `$M$-truncation of system \eqref{e8}' a  system,
obtained from \eqref{e8} by removing  the terms
$\RR^{atn}_{km}\,d\tilde\beta_m$ with $m>M$. The $M$-truncated and $\delta$-stopped
 system \eqref{e8} with $I=I_P$ is an
equation with progressively measurable coefficients,  Lipschitz continuous
 in $\vp^N$ (see \eqref{e5}).  So it has a unique strong solution $\vp^{N,M}$.
  Since
 $$
 \|\RR^{atn}(v)\|_{HS}\le C \,\|\tilde\RR(v)\|_{HS}= C\,\|\RR(v)\|_{HS},
 $$
 then all moments of the random variable $\sup_\tau\|\RR^{atn}(v(\tau))\|_{HS}$ are
 finite. Accordingly, the family of processes $(I_P,\vp^{N,M})\in h^I_P\times\T^N$,
 $M\ge1$,  is tight. Any limiting as $M\to\infty$
 measure  solves the martingale problem, corresponding to the $\delta$-stopped system
  \eqref{e3}${}_P$+\eqref{e8}${}_P$. So this is a law of a weak solution  $(I_P,\vp^N)$ of  that system
 (i.e., $(I_P,\vp^N)(\tau)$ satisfies the system with suitably chosen Wiener
 processes $\tilde\beta_m$). Accordingly,  we have
 constructed a desirable weak solution $(I,v^N)(\tau)$.
\end{proof}

We denote by $\PP^-_1$ the law of the constructed solution $(I,v^N)$. This is a measure
in $\hat\Omega$, supported by trajectories $(I,v^N)$ such that $v^N(\tau)$ is stopped at
$\tau=\thee$.
\smallskip

{\bf b)} Now we will extend  $\PP^-_1$ to a measure $\PP_1^+$ on $\hat\Omega$, supported by
trajectories $(I,v^N)$, where $v^N$ is stopped  at time $\theta_1^+$.

 Let us denote by  $\Theta=\Theta^{\theta_1^-}$  the operator which stops any
 continuous trajectory $\eta(\tau)$ at time
  $\tau=\thee$. That is, replaces it by $\eta(\tau\wedge\thee)$.

  Since
   $\DD(I^\nu_P(\cdot))\strela \DD(I_P(\cdot))$ as $\nu=\nu_j\to0$, then we can
   represent the laws
  $\PP^-_1$ and $\DD(v_P^\nu))$ by distributions of processes
  $(I'_P(\tau),{v'_P}^N(\tau))$ and ${v'_P}^\nu(\tau)$ such that
  $$
  I(\vb(\cdot)) \to I'_P(\cdot)\quad\text{as $\nu=\nu_j\to0$\;\; in}
  \quad \cH_I \quad\text{a.s. },
  $$
   and
  $$
  I(\va)\equiv {I'_P}^N\quad\text{for}\quad \tau\le\thee.
  $$
  Since $\vb(\tau,\omega)$, $0\le\tau\le T$,  is a diffusion process, we
 may replace it by a continuous process $\vc(\tau;\omega,\omega_1)$ on an extended
 probability space $\Omega\times\Omega_1$
  such that
\begin{enumerate}
\item
 $\
  \DD\,\vc=\DD\,\vb;
  $
\item  for $\tau\le\thee=\thee(\omega)$ we have $\vc=\vb$ (in particular, then $\vc$ is
independent from $\omega_1$);
\item  for $\tau\ge\thee$  the process  $\vc$
depends on $\omega$ only through the initial data
 $\vc(\thee,\omega,\omega_1) = \vb(\thee,\omega)$. For a fixed $\omega$ it
 satisfies \eqref{kdv_bir}${}_P$ with suitable Wiener processes
$\beta_j$'s, defined on the space $\Omega_1$.
\end{enumerate}

 Using a construction from \cite{KP08},
  presented in Appendix, for each
 $\omega$  we construct a continuous process ${(\bar w^\nu,\vcc)}(\tau;\omega,\omega_1)
\in h^p\times\R^{2N}$, $\tau\ge\thee$, $\omega_1\in\Omega_1$,  such that for
each $\omega$ we have
\smallskip

(i) law of the process $\bar w^\nu(\tau;\omega,\omega_1)$, $\tau\ge\thee$,
 $\omega_1\in\Omega_1$, is the same as of the process $w^\nu_P(\tau;\omega,\omega_1)$;
\smallskip

(ii)
$I(\vcc)=I^N(\bar w^\nu)$ for $\tau\ge\thee$ and
 $\vp \big({\vcc}(\thee)\big)=\vp(\va(\thee))$  a.s. in $\Omega_1$;
\smallskip

(iii) the law of the process ${\vcc}(\tau)$, $\tau\ge\thee$, is that of an
 It\^o process \begin{equation}\label{f1}
    dv^N=B^N(\tau)\,d\tau+a^N(\tau)\,dw(\tau),
\end{equation}
where for every $\tau$ the vector $B^N(\tau)$ and the matrix $a^N(\tau)$ satisfy
\begin{equation}\label{f2}
    |B^N(\tau)|\le C,\qquad C^{-1} I\le a^N (a^N)^t
(\tau)\le CI\quad \text{a.s},
\end{equation}
with some  $C=C(P,M)$.
\smallskip

Next for $\nu=\nu_j$ consider the process
$$
\xi^\nu_P(\tau)=\left(I_P^\nu(\tau)=I(\bar w^\nu(\tau)),\, \chi_{\tau\le\thee}\va
+\chi_{\tau>\thee}{\vcc}\right),\quad 0\le\tau\le T.
$$
Due to \eqref{v_est} and (iii) the family of laws $\{\DD(\xi_P^{\nu_j}), j\ge1\}$,
is  tight in the space
$C(0,T;h^I_p\times\R^{2N})$. Consider any limiting measure $\Pi$ (corresponding to
a suitable subsequence
$\nu'_j\to0$) and represent it by a process
$\tilde\xi_P(\tau)=(\tilde I_P(\tau), \tilde v^N_P(\tau))$, i.e.  $\DD\tilde\xi_P=\Pi$.
 Clearly,

(iv) $\DD(\tilde\xi_P)\mid_{\tau\le\thee}=\PP^-_1$,

(v) $\DD(\tilde I_P)=\DD(I_P)$.

Since any measure $\DD(\xi^\nu_P)$ is supported by the closed set, formed by all
trajectories $(I(\tau),v^N(\tau))$ satisfying $I^N\equiv I(v^N)$, then the
limiting measure $\Pi$ also is supported by it. So the process $\tilde\xi_P$
satisfies

(vi) $I(\tilde v_P^N(\tau))\equiv \tilde I^N_P(\tau)$   a.s.

Moreover, for the same reasons as in Appendix  the law of the
limiting process $\tilde v^N_P(\tau)$, $\tau\ge\thee$, is that of an
It\^o process \eqref{f1}, \eqref{f2}. (Note that  for $\tau\ge\thee$ the  process
$\tilde v^N_P$ {\it is not} a solution  of \eqref{e7}).

Now we set
$$
\PP^+_1=\Theta^{\theta_1^+}\circ \DD(\tilde\xi_P).
$$

\smallskip

{\bf c)} The constructed measure $\PP_1^+$ gives us distribution of a process
$(I(\tau), v^N(\tau))$ for ${\tau\le\theta_1^+}$.
 Next we solve eq. \eqref{e7}${}_P$ on the interval
$\Lambda_1=[\theta_1^+,\theta_2^-]$ with the initial data
$(I(\theta_2^-), v^N(\theta_2^-)$ and iterate the construction.

It is easy to see that a.s. the sequence $\theta_j^\pm$ stabilises at $\tau=T$
 after a finite
(random) number of steps. Accordingly the sequence of measures $\PP_j^\pm$
 converges to a limiting measure $\PP_\delta$ on $\hat\Omega$.
\smallskip

{\bf d)} On the space $\tilde\Omega$, given the measure $\PP_\delta$, consider
 the natural process
which we denote $\xi_\delta(\tau)=(I_\delta(\tau), v^N_\delta(\tau))$. We have
\begin{enumerate}
\item  $\DD(I_\delta(\cdot))=\DD(I_P)$,

\item $I(v^N_\delta(\cdot))\equiv I^N_\delta$  a.s.,

\item for $\tau\in\Lambda$ the process $\xi_\delta$ is a weak solution of
 \eqref{e7}${}_P$, while for $\tau\in\Delta$ the process $v^N_\delta(\tau)$
 is distributed as an It\^o process
\eqref{f1}.
\end{enumerate}
\medskip

\noindent
{\bf Step 4.} {\it Limit $\delta\to0$}.

Due to 1-3 the set of measures $\{\PP_\delta, 0<\delta\le1\}$ is tight.
 Let $\PP_P$ be any limiting measure as $\delta\to0$. Clearly it meets 1 and 2
 above.

\begin{lemma}\label{l.lim}
The measure $\PP_P$ is a solution of the martingale problem for equation
 \eqref{e7}${}_P$.
\end{lemma}
The lemma is proved in the next subsection.

\medskip

\noindent
{\bf Step 5.} {\it Limit $P\to\infty$}.

Due to 1, 2 above, relations
\eqref{v_est} and Lemma~\ref{l.lim} the set of measures $\PP_P$,
$P\to\infty$, is tight. Consider any liming measure $\PP^N$ for this family.
 Repeating in
a simpler way the proof of Lemma~\ref{l.lim} we find that $\PP^N$ solves the martingale
problem \eqref{e7}. It still satisfies 1 and~2 (see Step~3d)\,). Let
 $(I(\tau), v^N(\tau))$
be a weak solution for \eqref{e7} such that its law equals $\PP^N$. Denote
by $\vN(\tau)$ the
process $(v^N(\tau), V^{>N}(\tau))$ and denote by $\mu^N$ its law in the space $\cH_v$
(see \eqref{notat}).

\medskip

\noindent
{\bf Step 6.} {\it Limit $N\to\infty$}.

Due to  \eqref{est} the family of measures $\{\mu^N\}$ is tight in $\cH_v$.
 Let $N_j\to\infty$ be a sequence such that $\mu^{N_j}\strela\mu$.

 The process $\vN(\tau)$ satisfies equations \eqref{e4}${}_{1\le k\le N}$ with suitable
standard  independent Wiener processes $\tilde\beta_m(\tau)$. Due to
 Lemma~\ref{l.equiv}
 and a remark, made after it, the process also satisfies
 equations \eqref{e44}${}_{1\le k\le N}$.
 Repeating again the proof of Lemma~\ref{l.lim} we see that $\mu$ is a martingale
 solution of the system \eqref{e44}${}_{1\le k\le N}$ for any $N\ge1$. Hence, $\mu$ is a
 martingale solution of \eqref{e44} and of \eqref{lif_sy}.
  Let $v(\tau)$ be a
 corresponding weak solution of \eqref{e44}, $\DD(v(\cdot))=\mu$. As
 $\mu^{N_j}\strela\mu$, then the process $v$ satisfies assertions  $i)$
 and $ii)$ in Theorem~\ref{l_lif} and the theorem is proved.
\end{proof}

\subsection{Proof of Lemma \ref{l.lim}.}

Consider the space $\hat\Omega$ with the natural filtration $\FF_\tau$, provide
 it with a measure
$\PP_\delta$ and, as usual, complete the sigma-algebras  $\FF_\tau$ with respect
 to this measure.
As before we denote by $\xi_\delta(\tau)=(I_\delta(\tau), v^N_\delta(\tau),
 0\le\tau\le T)$, the
natural process on $\hat\Omega$.

i) For $k\ge1$  consider the process $I_{\delta k}(\tau)$.
It satisfies the $I_k$-equation in
 \eqref{e7}${}_P$:
 \begin{equation}\label{g1}
    dI_k=F_k^P(\tau,I)\,d\tau + \sum M^P_{km}(\tau,I)\,d\tilde\beta_m(\tau).
 \end{equation}
 Here $F_k^P$ equals  $F_k$ for $\tau\le\tau_P$ and equals
 $\tfrac{1}{2}b_k^2$ $\tau>\tau_P$, while
 $M^P_{km}$  equals
 $M_{km}$   for  $\tau\le\tau_P$ and equals $b_k\sqrt{2I_k}$ for
  $\tau>\tau_P$, cf.~\eqref{stop2}. For each $\delta>0$ and
   any $k$ the process
 $ \chi^I_k(\tau)=I_k(\tau)-\int_0^\tau F_k^P(s,I(s))\,ds$
  is an  $\PP_\delta$-martingale. Due to \eqref{est} the $L_2$-norm of these
 martingales are bounded uniformly in $\tau$ and $\delta$. Since
 $\PP_\delta\strela\PP_P$ and the laws of the processes $\chi_k^I$, corresponding
to $\delta\in(0,1]$ are tight in $C[0,T]$, then
   $\chi^I_k(\tau)$ also is an $\PP_P$-martingale.
  \smallskip

  ii) Consider a process $\vv_{\delta k}$, $1\le k\le N$. It satisfies
 \eqref{e7}${}_P$ for
  $\tau\in\Lambda$ and satisfies the $k$-th equation in \eqref{f1}
 for $\tau\in\Delta$, where the
  vector $B^N(\tau)$ and the operator $a^N(\tau)$, $\tau\in\Delta$, meet
 the estimates \eqref{f2}.
  So $\vv_{\delta k}$ satisfies the It\^o equation
  \begin{equation}\label{g2}
    \begin{split}
    d\vv_k(\tau)=&\left(\chi_{\tau\in\Lambda}R_k^P(\tau,v)+\chi_{\tau\in\Delta}
    B_k^N(\tau)\right)d\tau  \\
    +&\chi_{\tau\in\Lambda}\sum_m\tilde\RR^P_{km}(\tau,v)\,d\tilde\beta_m(\tau)
    +\chi_{\tau\in\Delta}\sum_r a^N_{kr}(\tau)\,dw_r(\tau)\\
    =:&\,A^\delta_k(\tau)\,d\tau + \sum_{m\ge1}G^\delta_{km}(\tau,v)\,d\tilde\beta_m(\tau)
    +\sum_{r=1}^{2N} C^\delta_{kr}(\tau)\,dw_r(\tau).
    \end{split}
  \end{equation}
  Note that the random dispersion matrices $G^\delta(\tau)$ and $C^\delta(\tau)$ are
  supported  by non-inter\-sec\-ting  random   time-sets.

For any $\delta>0$ the process
 $ \chi^\delta_k(\tau)=\vv_k(\tau)-\int_0^\tau A_k^\delta(s)\,ds\in\R^2$
is an $\PP_\delta$-martingale. Let us compare $\int A_k^\delta\,ds$ with the
corresponding term in \eqref{e7}${}_P$. For this end we consider the quantity
\begin{equation}\label{g3}
    \begin{split}
    &\E \sup_{0\le\tau\le T}\left|\int_0^\tau A_k^\delta(s)\,ds -\int_0^\tau
 R_k^P(s,v(s))\,ds
    \right|\\
    &\le \E\int_\Delta\left|R^P_k(s,v(s)) \right|\,ds+\E\int_\Delta|B_k^N(s)|\,ds
    =:\Upsilon_1+\Upsilon_2.
    \end{split}
\end{equation}
By \eqref{v_est} and\eqref{e120},
$$
\Upsilon_1^2\le\E\int_0^T|R_k^P|^2\,ds \cdot \E\int_0^T\chi_\Delta(s)\,ds\le
 C(P)\,o_\delta(1).
$$
Similar $\Upsilon_2\le C(P)\,o_\delta(1)$. So \eqref{g3} goes to zero with
  $\delta$. Since the
$L_2$-norms of the martingales $ \chi^\delta_k$ are uniformly bounded and their laws
are tight in $C(0,T;\R^2)$,  then
 $ \chi^0_k(\tau)=\vv_k(\tau)-\int_0^\tau R_k^P(s)\,ds$
is an $\PP_P$-martingale.
Indeed, let us take any $0\le\tau_1\le\tau_2\le T$ and let $\Phi\in C_b(\hat\Omega)$
be any
 function such that $\Phi(\xi(\cdot))$ depends only on $\xi(\tau)_{\,0\le\tau\le\tau_1}$.
 We
 have to show that
 \begin{equation}\label{h00}
    \E^{\PP_P}\left((\chi^0_k(\tau_2)-\chi^0_k(\tau_1))\Phi(\xi)\right)=0.
 \end{equation}
 The l.h.s. equals
 \begin{equation*}
    \begin{split}
 &\lim_{\delta\to0}\E^{\PP_\delta}\left((\chi^0_k(\tau_2)-
 \chi^0_k(\tau_1))\Phi(\xi)\right)\\
 &=\lim_{\delta\to0}\E^{\PP_\delta}\left(\Phi(\xi)\Big(\vv_k(\tau_2)-
 \vv_k(\tau_1)-\int_{\tau_1}^{\tau_2}R_k^P(s)\,ds\Big)\right)\\
 &=\lim_{\delta\to0}\E^{\PP_\delta}\left(\Phi(\xi)\int_{\tau_1}^{\tau_2}
 \big(A_k^\delta(s) -R_k^P(s)\big)ds\right)
    \end{split}
 \end{equation*}
(we use that  $\chi^\delta_k$ is a $\PP_\delta$-martingale).  The r.h.s. is
 $$
 \le C \lim_{\delta\to0}\E^{\PP_\delta}\sup_\tau\left|\int_0^\tau(A_k^\delta(s)-R_k^P(s))\,ds
 \right|\le C\lim_{\delta\to0}(\Upsilon_1-\Upsilon_2)=0.
 $$
So \eqref{h00} is established.
\smallskip

iii) For the same reasons as in i), for each $k$ and $l$ the process
$$
\chi^I_k(\tau)\chi_l^I(\tau)-\frac12 \int_0^\tau\sum_mM^P_{km}(s,I(s))M^P_{lm}(s,I(s))\,ds
$$
is an $\PP_P$-martingale.
\smallskip

iv) Due to \eqref{g2}, for any $\delta$ and any $k,l\le N$ the process
\begin{equation*}
\begin{split}
\chi^\delta_k(\tau)\chi_l^\delta(\tau)&-\frac12
 \int_0^\tau\left(\sum_mG^\delta_{km}G^\delta_{lm}+
C^\delta_{km}C^\delta_{lm}\right)\,ds\\
&=:\chi^\delta_k(\tau)\chi_l^\delta(\tau)-\frac12\int_0^\tau\left(X_{kl}(s)+Y_{kl}(s)
\right)ds
\end{split}
\end{equation*}
is a $\PP_\delta$-martingale. We compare it with the corresponding expression
for eq. \eqref{e7}${}_P$.
To do this we first consider the expression
\begin{equation}
\begin{split}
&\E\sup_{0\le\tau\le T}\left|\frac12\int_0^\tau\left(\sum_m\tilde\RR^P_{km}
\tilde\RR^P_{lm}
-X_{kl}-Y_{kl}  \right)ds
\right|\\
&\le\E\,\frac12\int_0^\T \left|\sum_m\tilde\RR^P_{km}
\tilde\RR^P_{lm}\right|\chi_{s\in\Delta}\,ds
+\E\,\frac12\int_0^T \left|\sum_ma^N_{km}a^N_{lm}\right|\chi_{s\in\Delta}\,ds.
\end{split}
\end{equation}
As in ii), the r.h.s. goes to zero with $\delta$. Hence,
$\
\chi^0_k(\tau)\chi_l^0(\tau)-\frac12 \int_0^\tau\tilde\RR^P_{km}\tilde\RR^P_{lm}\,ds
$
is an $\PP_P$-martingale by the same arguments that prove \eqref{h00}.
\smallskip

v) Finally consider the $I,v$-correlation. For $k\ge1$ and $1\le l\le N$ the process
\begin{equation*}
    \begin{split}
\R^2\ni\chi^I_k(\tau)\chi_l^\delta(\tau)-\frac12 \int_0^\tau
\sum_mM^P_{km}G^\delta_{lm}\,ds-&
\frac12 \int_0^\tau\sum_{m\ge1}\sum_{r=1}^{2N}
 M^P_{km}C^\delta_{lr}\,d[\tilde\beta_m,w_r](s)\\
=:&\chi^I_k(\tau)\chi_l^\delta(\tau)-\frac12 \int_0^\tau\Xi_{kl}^\delta(s)\,ds
    \end{split}
\end{equation*}
is an $\PP_\delta$ martingale. We know that
\begin{enumerate}
\item  the matrix $\frac{d}{ds}[\tilde\beta_m,w_r](s)$ is constant in $s$ and
is such that $l_2$-norms of all its columns and rows are bounded by one;
\item $\|M^P\|_{HS},\ \|C^\delta\|_{HS}\le C(P)$ for all $\delta$.
\end{enumerate}
Therefore
$$
\left|
\sum_{m\ge1}\sum_{r=1}^{2N}  M^P_{km}C^\delta_{lr}\, \frac{d}{ds}
 [\tilde\beta_m,w_r](s)\right|
\le C_1(P)\,.
$$
Now repeating once again the arguments in ii) we find that
$$
\E\sup_{0\le\tau\le T}\frac12 \left|\int_0^\tau
\left(\sum_mM^P_{km} \tilde\RR^P_{lm}-\Xi_{kl}^\delta\right)ds
\right|\to0
$$
as $\delta\to0$. Therefore the process
$\
\chi^I_k(\tau)\chi_l^\delta(\tau)-\frac12 \int_0^\tau\sum_mM^P_{km} \tilde\RR^P_{lm}\,ds\
$
 is an $\PP_P$-martingale.
 \smallskip

 Due to i)-v) the measure $\PP_P$ is a martingale solution for eq.~\eqref{e7}${}_P$.\qed

\section{Uniqueness of solution}\label{s_uniq}

In this section we will show that a regular solution of the effective equation
(\ref{lif_sy}) (i.e. a solution that satisfies
estimates (\ref{est}))  is unique. Namely, we will prove the following result:

\begin{theorem}\label{t_unistr}
If $v^1(\tau)$ and $v^2(\tau)$ are strong regular solutions of (\ref{lif_sy}
) with\\ $v^1 (0)=v^2(0)$ a.s., then $v^1(\cdot)=v^2(\cdot)$ a.s.
\end{theorem}

Using the  Yamada-Watanabe arguments (see, for instance, \cite{KaSh}), we conclude
that  uniqueness of a strong regular solution for (\ref{lif_sy}) implies
uniqueness of a  regular weak  solution.  So we get
\begin{corollary}
If $v^1$ and $v^2$ are  regular weak  solutions of equation  (\ref{lif_sy}) such
 that $\DD(v^1(0))=\DD(v^2(0))$, then $\DD(v^1(\cdot))=\DD(v^2(\cdot))$.
\end{corollary}

\begin{corollary}
Under the assumptions of Theorem~\ref{l_lif} the law of a lifting $v(\tau)$ is defined
in a unique way.
\end{corollary}

Evoking Theorem~\ref{l_lif} we obtain

\begin{corollary}\label{c_uniweak}
Let $I^1(\tau)$ and $I^2(\tau)$ be weak  regular solutions of
  \eqref{aveq_i}, \eqref{IC}   as in Theorem~\ref{t_comp1} (i.e. these are
   two limiting points of the family
  of measures $\DD(I^\nu(\cdot))$). Then their laws coincide.
\end{corollary}

These results and  Theorem~\ref{t_comp1} jointly imply

\begin{theorem}\label{t_final}
The action vector $I^\nu(\cdot)$ converges in law in the space $\cH_I$ to a
 regular weak  solution $I^0(\cdot)$  of (\ref{aveq_i}), \eqref{IC}.
Moreover, the law of $I^0$ equals $I\circ\DD(v(\cdot))$, where $v(\tau)$ is a
unique regular weak solution of \eqref{lif_sy} such that $v(0)=V_\vt(I_0)$. Here
$\vt$ is any fixed vector from the torus $\T^\infty$.
 \end{theorem}

\begin{proof}[Proof of Theorem \ref{t_unistr}]
Denote by $(\cdot,\cdot)_{0}$ the inner product in $h^0$.
For a fixed $\kappa>0$ we introduce the stopping time $ {\nnu}$:
$$
{\nnu}=\min\{\tau\le T\,:\,|v^1(\tau)|_{h^2}\vee v^2(\tau)|_{h^2} =\kappa\}
$$
(if the set is empty we set ${\nnu}=T$). Due to (\ref{bound})
$$
{\bf P}\{{\nnu}<T\}\le c\kappa^{-1}.
$$
Denote
$$
v^j_\kappa(\tau)=v^j(\tau\wedge{\nnu}),\qquad w(\tau)=v_k^1(\tau)-v_k^2(\tau).
$$
To prove the theorem  it suffices to show that $w(\tau)=0$ a.s.,
 for each $\kappa>0$.

We have
$$
dw_k(\tau)=\chi_{\tau<{\nnu}}\Big\{[R^1_k(v^1_\kappa)-R^1_k(v^2_\kappa)]d\tau-
[R^2_k(v^1_\kappa)-R^2_k(v^2_\kappa)]d\tau
$$
$$
+\sum\limits_{l\ge1}\,\int\limits_{\mathbb T^\infty}[{\cal R}(k;l,\theta)(v^1_\kappa)-
{\cal R}(k;l,\theta)(v^2_\kappa)]\,d\bb_{l,\theta}d\theta  \Big\}
$$
Application of the  It\^o formula yields
\begin{equation*}
\begin{split}
\E\,|w(\tau)|_0^2= & \E\int\limits_0^{\tau\wedge{\nnu}}\left(w(s),
[R^1(v^1_\kappa)-R^1(v^2_\kappa)] \right)_0ds\\
+&
\E\int\limits_0^{\tau\wedge{\nnu}}\big(w(s),[R^2(v^1_\kappa)-R^2(v^2_\kappa)]
\big)_0ds\\
+&\frac{1}{2}\E\int\limits_0^{\tau\wedge{\nnu}}\sum\limits_{l\ge1}
\int\limits_{\mathbb T^\infty}|{\cal R}(\cdot,l,\theta)(v^1_\kappa)-
{\cal R}(\cdot,l,\theta)(v^2_\kappa)|_0^2\,d\theta ds\equiv\Xi_1+\Xi_2+\Xi_3.
\end{split}
\end{equation*}
We will estimate the three terms in the r.h.s.
and
start with  the term $\Xi_3$. By the definition of $\mathcal{R}(k;l,\theta)(v)$
and due to item 4 of Theorem \ref{t_kp08} we have
$$
|\mathcal{R}(\cdot,l,\theta)(v^1_\kappa(s))- \mathcal{R}(\cdot,l,\theta)
(v^2_\kappa(s))|_0^2\le C_N\kappa^{n_0}l^{-N}|w(s)|_0^2
$$
for any $N\in\mathbb Z^+$, with a suitable $n_0\in\mathbb Z^+$. Therefore,
$$
\Xi_3\le C\kappa^{n_0}\E\int\limits_0^{\tau\wedge{\nnu}}|w(s)|_0^2\,ds.
$$
For similar reasons
$\
\Xi_2\le C\kappa^{n_0}\E\int\limits_0^{\tau\wedge{\nnu}}|w(s)|_0^2\,ds.
$

Estimating the term $\Xi_1$ is more complicated since the map $v\mapsto R^1(v)$
is unbounded in every space $h^p$.
We remind that $\LL^{-1}:=d\Psi(0)$ is the diagonal operator
$$
{\cal L}^{-1}\Big(\sum_s u_sf_s\Big)=
v,\quad v_s=|s|^{-1/2}u_s\;\;\forall\,s\in\Z_0,
$$
and introduce
$\Psi_0(u)=\Psi(u)-{\cal L}^{-1}u$. 
 According to Proposition (\ref{amplif_1_1}), $\Psi_0$ defines analytic maps
$H^m\,\mapsto\,h^{m+1}$, $m\ge 0$. We denote by $G$ the inverse map $G=\Psi^{-1}$.
 Then $G(v)={\cal L}(v)+G_0(v)$, where $G_0\,:\,h^m\longrightarrow H^{m+1}$
 is analytic for any $m\ge 0$. Finally, denote $R^1(v)-\widehat\Delta v=R^0(v)$,
 where $\hD$ is the Fourier-image of the Laplacian:
 $\widehat\Delta v=v'$, where $\vv'_j=-j^2\vv_j$, $\forall\,j$.

\begin{lemma}\label{l4.1}
For any $m\ge1$ the map $R^0:h^m\to h^{m-1}$ is analytic.
\end{lemma}

So  the effective equation \eqref{lif_sy} is a quasilinear stochastic heat equation.

\begin{proof}
We have
$$
R^1(v)=\int_{\T^\infty}\Phi_{-\theta}\LL^{-1}\Delta(G\Phi_\theta v)\,d\,\theta+
\int_{\T^\infty}\Phi_{-\theta}d\Psi_0(G\Phi_\theta v)\Delta(G\Phi_\theta v)
\,d\,\theta\,.
$$
The first integrand equals
$$
\Phi_{-\theta}\LL^{-1}\Delta\LL\Phi_\theta v+\Phi_{-\theta}\LL^{-1}\Delta (G_0 \Phi_\theta v)=
\hD v  +\Phi_{-\theta}\LL^{-1}\Delta( G_0 \Phi_\theta v)\,
$$
since $\LL^{-1}\Delta\LL\Phi_\theta=\hD$ and $\hD$ commutes with
 the operators $\Phi_\theta$.

We have $d\Psi_0(u_\theta):h^m\to h^{m+1}$.
 Since the map $\Psi$ is symplectic, then
also $d\Psi_0(u_\theta):h^r\to h^{r+1}$ for $-m-2\le r\le m$ (cf. Proposition~1.4
 in \cite{K2}).  So for any $\theta$ the
second integrand defines an analytic map $h^m\to h^{m-1}$. Now the assertion follows.
\end{proof}

By this lemma with $m=1$
\begin{equation*}
\begin{split}
\Xi_1&=\E\int_0^{\tau\wedge\Theta}\left(-|w(s)|_1^2+\big(w(s),R^0(v^1_\kappa)
-R^0(v^2_\kappa)\big)_0
\right)ds\\
&
\le \E\int_0^{\tau\wedge\Theta}\left(-|w(s)|^2_1+C_\kappa |w(s)|_0|w(s)|_1\right)ds
\le \E\, C'_\kappa \int_0^{\tau\wedge\Theta}|w(s)|^2_0\,ds.
\end{split}
\end{equation*}

Combining the obtained estimates for $\Xi_1$, $\Xi_2$ and $\Xi_3$, we arrive at
the inequality
$$
\E|w(\tau)|_0^2\le C_\kappa^1\int\limits_0^\tau\E|w(s)|_0^2ds.
$$
Since $\E|w(0)|_0^2=0$, then $\E|w(\tau)|_0^2=0$ for all $\tau$.
This completes the proof of Theorem~\ref{t_unistr}.
\end{proof}

\section{Limiting  joint  distribution of action-angles} \label{s_act-angl}

For a solution $u^\nu(t)$ of \eqref{kdv}, \eqref{k0} we denote by
$I^\nu(\tau)=I(v^\nu(\tau))$ and $\vp^\nu(\tau)=\vp(v^\nu(\tau))$ its actions
and angles, written in the slow time $\tau$. Theorem~\ref{t_final} describes
limiting behaviour of $\DD I^\nu$ as $\nu\to0$. In this section we study joint distribution
of $I^\nu(\tau)$ and $\vp^\nu(\tau)$,
 mollified in $\tau$. That is,  we study  the measures
$
\mu^\nu_f=\int_0^Tf(s)\DD\Big(   I^\nu(s),\vp^\nu(s)  \Big)\,ds
$
on the space $h_I^p\times\T^\infty$, where $f\ge0$ is a continuous function
such that $\int_0^T f=1$.

\begin{theorem}\label{t_v}
As $\nu\to0$,
\begin{equation}\label{h9}
\mu^\nu_f\strela \Big(\int_0^Tf(s)\DD(I^0(s))\,ds\Big)\times d\vp.
\end{equation}
In particular,$\ \int_0^Tf(s)\DD(\vp^\nu(s))\,ds \strela d\vp$.
 \end{theorem}
\begin{proof}
Let us first replace $f(\tau)$ with a characteristic function
$$
\bar f(\tau)=\frac1{T_2-T_1}\,\chi_{\{T_1\le\tau\le T_2\}}\,,\qquad0\le T_1<T_2\le T.
$$
Due to \eqref{xxx} the family of measures $\{\mu_{\bar f}^\nu, \nu>0\}$ is tight in
$h_I^p\times \T^\infty$. Consider any limiting measure
 $\mu_{\bar f}^{\nu_j}\strela \mu_{\bar f}$.

Let $F(I,\vp)=F^0(I^m,\vp^m)$, where $F^0$ is a bounded Lipschitz function on
$\R^m_+\times\T^m$. We claim that
\begin{equation}\label{h10}
\frac1{T_2-T_1}\,\int_{T_1}^{T_2}\E F(I^\nu(s),\vp^\nu(s))\,ds\to
\frac1{T_2-T_1}\,\int_{T_1}^{T_2}\E \lan F\ran(I^0(s))\,ds\quad {\rm as}\quad \nu\to0.
\end{equation}
Indeed, due to Theorem~\ref{t_final} we have
$$
\frac1{T_2-T_1}\,\int_{T_1}^{T_2}\E \lan F\ran(I^\nu(s))\,ds\to
\frac1{T_2-T_1}\,\int_{T_1}^{T_2}\E \lan F\ran(I^0(s))\,ds \quad {\rm as}\quad \nu\to0.
$$
So
\eqref{h10} would follow if we prove the
convergence
\begin{equation}\label{h11}
\E\left|\int_0^\tau F(I^\nu(s),\vp^\nu(s))-\lan F\ran(I^\nu(s))\right|\,ds\to0
\quad{\rm as}\quad\nu\to0,
\end{equation}
for any $\tau$. But \eqref{h11} is established in \cite{KP08} (see there (6.9) and
below) for $F^0(I^m,\vp^m)=F_k(I^m,0;\vp^m,0))$, where $F_k(I,\vp)$ is the drift in
eq.~\eqref{eq_for_i}. The arguments in \cite{KP08} are general and apply to any
bounded Lipschitz function $F^0$.

Relation \eqref{h10} implies that
$
\mu_{\bar f}=
\Big((T_2-T_1)^{-1}\,\int_{T_1}^{T_2}\DD(I^0(s))\,ds\Big)\times d\vp
$.
 So \eqref{h9} is established for characteristic functions. Accordingly, \eqref{h9}
 holds, firstly, for piece-wise constant functions $f(\tau)$ with finitely many
discontinuities and, secondly, for continuous functions.
\end{proof}

\section{Appendices}

\subsection {Whitham averaging}   The $n$-gap solutions of the KdV equation under the zero-meanvalue
periodic boundary condition have the form \eqref{k2}, where $0=I_{n+1}=I_{n+2}=\dots$.  They
depend on the initial phase $\vp\in\T^n$ and the $n$-dimensional action $I^n\in\R^n$.
 These solutions form
a subset of the bigger family of space-quasiperiodic $n$-gap solution which may be written as
$\Theta^n(Kx+W t+\vp;w)$. Here the parameter $w$ has dimension $2n+1$, $\Theta^n$ is an analytic
function  on $\T^n\times\R^{2n+1}$ and the vectors $K,W\in\R^n$ depend on $w$. See
in \cite{ZMNP, DN89, LLV93, K2}.

Denote by $X=\nu x$ and $T=\nu t$ slow space\,-  and time-variables. We want to solve either
the KdV itself, or some its $\nu$-perturbation (say, eq.~\eqref{kdv}${}_{\eta=0}$) in the space of
functions, bounded as $|x|\to\infty$ (not necessarily periodic in $x$). We are looking
for solutions with the
initial data
$$ u_0( x)=\Theta^n(Kx+\vp_0;w_0(X)),$$
where $w_0(X)\in\R^{2n+1}$ is a given vector-function. Assuming that a solution $u(t,x)$ exists,
decomposes in asymptotical  series in $\nu$ and that the leading term may be written as
\begin{equation}\label{W}
u^0(t,x)=\Theta^n(Kx+ W t+\vp_0; w(T,X)),
\end{equation}
Whitham shown that $w(T,X)$ has to satisfy a nonlinear hyperbolic
 system, known now as
the {\it Whitham equations}. In the last 40 years much attention
 was given to the Whitham
equations and Whitham averaging (i.e. to the claim that an exact
 solution $u(t,x)$ may be
written as
 $u=u^0(t,x)+o(1)$, where $u^0$ has the form \eqref{W}). Many results
 were obtained
for the Whitham  equations for KdV and for other integrable
systems,
 e.g. see \cite{ZMNP, Kri88, DN89} (we note that
in the last section of
 \cite{DN89} the authors discuss the damped equation \eqref{kdv}${}_{\eta=0}$).
In these works the Whitham equations are postulated as a first principle,
without precise statements on their connection with the original problem.
Rigorous results on this connection, i.e. results on the Whitham averaging,
 are very few, and these are examples
 rather than general
theorems since they apply to {\it some} initial data and hold in
 {\it some} domains in the
space-time $\R^2$, see in   \cite{ LLV93}.

 In the spirit of the Whitham theory our results may be casted
 in the following way. Consider
 a perturbed KdV equation
 \begin{equation}\label{pkdv}
\dot u+u_{xxx}-6uu_x=\nu f(u,u_x,u_{xx}),
\end{equation}
and take initial condition $u_0(x)$ of the form above with arbitrary $n$,
where $w_0$ is an $x$-independent  constant  such that
$u_0(x)$ is 2$\pi$-periodic with zero mean-value.
 Let us write $u_0$ as a periodic $\infty$-gap potential
$
u_0(x)=\Theta^\infty(Kx+\vp_0;I_0),
$
 where $\Theta^\infty:\T^\infty \times\R_+^\infty\to\R$ and now $K\in\Z^\infty$, $\vp_0\in\T^\infty$
(see \cite{McKT76} for a
theory of $\infty$-gap potentials). We may write a solution of \eqref{pkdv} as
$
u^\nu(t,x)=\Theta^\infty(Kx+\vp^\nu(\tau);I^\nu(\tau))$, $\tau=\nu t
$,
with unknown phases $\vp^\nu\in\T^\infty$ and actions $I^\nu\in\R_+^\infty$.  The
main task is to recover the  actions. To do this we write the effective equations for
$I(\tau)$, corresponding to \eqref{pkdv}. Namely, we rewrite  \eqref{pkdv} using the
non-linear Fourier transform $\Psi$, pass to the slow time $\tau$, delete  from the obtained
$v$-equation the KdV vector-field
$d\Psi\circ V$ and apply to the rest the averaging \eqref{eff}. We claim that for some
classes of perturbed KdV equations the vector  $I^0(\tau)=\pi_I(v(\tau))$, where $v$
 solves the  effective equations, well approximates $I^\nu(\tau)$
with small $\nu$. Our work justifies this claim for the damped-driven perturbations
\eqref{kdv} in the sense that the convergence \eqref{k5} holds.

This special case of the Whitham averaging deals with perturbations of solutions for KdV
which fast oscillate in time (since we write them using  the slow time $\tau$), while the
general case treats solutions which fast oscillate both in the slow time $T$ and slow space $X$. The
effective equations serve to find approximately the action vector $I^\nu(\tau)\in\R^\infty_+$
which represents a space-periodic solution for \eqref{pkdv} as an infinite-gap potential
$\Theta^\infty(Kx+\vp^\nu(\tau);I^\nu(\tau))$. They play a role, similar to that of the
Whitham equations, serving to find the parameter $w(T,X)\in\R^{2n+1}$, describing $n$-gap potentials
\eqref{W} which approximate (non-periodic) solutions.

\subsection{   Lemma 4.3 from \cite{KP08}}
Below we present a construction from \cite{KP08},    used essentially in
 Section~\ref{s3}.

For $\tau\ge \theta'\ge0$ consider a solution $v(\tau)=v_P^\nu(\tau)$ of equation
 \eqref{kdv_bir}${}_P$.
For any $N\in\N$  we will construct a process
$(\bar v,\tilde v^N)(\tau)\in h^p\times\R^{2N}$, $\tau\ge\theta'$,
 such that
\begin{enumerate}
\item $\DD(\bar v(\cdot))=\DD(v(\cdot))$;
\item  $ I(\tilde v^N(\tau))\equiv I^N(v(\tau))$, a.s.;
\item $\vp(\tilde v^N(\theta'))=\vp^0$, where $\vp^0$ is a given vector in $\T^N$;
\item the process $\tilde v^N(\tau)$ satisfies certain estimates  uniformly in $\nu$.
\end{enumerate}

For $\eta_1,\eta_2\in\R^2\setminus\{0\}$ we denote  by $U(\eta_1,\eta_2)$ the
 operator in
$SO(2)$ such that $U(\eta_1,\eta_2)\frac{\eta_1}{|\eta_1|}= \frac{\eta_2}{|\eta_2|}$.
If $\eta_1=0$ or $\eta_2=0$,  we set  $U(\eta_1,\eta_2)=\,$id.

 Let us
abbreviate  in eq.~\eqref{kdv_bir}${}_P$ \ $(P^1_k(v)+P^2_k(v))^P=A_k^P(v)$. Then
the equation takes the form
\begin{equation}\label{7.1}
d\tilde\vv_k= \big(\nu^{-1}d\Psi_k(u)V(u)\big)^P d\tau+
 A_k^P(v)\,d\tau +
\sum_{j\ge1} B_{kj}^P(v)\,d\bb_j(\tau),\quad
1\le k\le N.
\end{equation}
 For $1\le k\le N$ we
 introduce the functions
$$
\tilde A_k(\tilde\vv_k,v)=U(\tilde\vv_k,\vv_k) A_k^P(v),\qquad
 \tilde B_{kj}(\tilde\vv_k,v)=U(\tilde\vv_k,\vv_k) B_{kj}^P(v),
$$
and define additional stochastic system for a vector
$\tilde v^N=(\tilde\vv_1,\dots,\tilde\vv_N)\in\R^{2N}$:
\begin{equation}\label{7.2}
d\tilde\vv_k= \tilde A_k(\tilde\vv_k,v)\,d\tau +
\sum_{j\ge1}\tilde B_{kj}(\tilde\vv_k,v)\,d\bb_j(\tau),\quad
1\le k\le N.
\end{equation}
Consider the system of equations \eqref{7.1}, \eqref{7.2}, where
  $\tau\ge\theta'$, with the initial
condition
\begin{equation}\label{h1}
\tilde v^N(\theta')=V^N_{\vp^0}\big(I(v^N(\theta'))\big)
\end{equation}
and with the given
 $v(\theta')$. It has a unique strong solution, defined while
$$
|\vv_k|, \; |\tilde\vv_k|\ge c>0\quad \forall\,k\le N,
$$
for any fixed $c>0$.

Denote $[(v, \tilde v)] =\left( \min_{1\le j\le N}\12|\vv_j|^2 \right) \wedge
\left( \min_{1\le j\le N}\12|\tilde \vv_j|^2 \right)$. Fix any $\gamma\in(0,\tfrac14]$
 and define
stopping times $\tau_j^\pm\in[\theta', T]$, $\dots,\tau_j^-<\tau_j^+<\tau^-_{j+1}<\dots$,
 as at
Step~3 in Section~3.1. Namely,
\begin{itemize}
\item If $[(\vv_0,\vv_0)]\le\gamma$, then $\tau_1^-=0$. Otherwise $\tau^+_0=0$.
\item If $\tau_j^-$ is defined, then $\tau_j^+$ is the first moment after
  $\tau_j^-$  when
$[(v(\tau),\tilde v(\tau))]\ge 2\gamma$ (if this never happens, then $\tau_j^+=T$).
\item If $\tau_j^+$ is defined, then $\tau^-_{j+1}$ is the first moment
 after $\tau_j^+$ when
$[(v,\tilde v)]\le\gamma$.
\end{itemize}
Next for $0<\gamma\le\tfrac14$
 we construct a continuous process $(v(\tau), \tilde v^{\gamma N}(\tau))$,
 $\tau\ge\theta'$,
  where  $v(\tau)\equiv v_P^\nu(\tau)$,
 $\tilde v^N(\theta')$ is given (see \eqref{h1}),
 and
for $\tau>\theta'$  the process $\tilde v^{\gamma N}$ is defined as follows:

i) If $\tilde v^{\gamma N}(\tau_j^+)$ is known, then we extend $\tilde v^{\gamma N}$ to the
segment $\Delta_j:=[\tau_j^+,\tau_{j+1}^-]$ in such a way that $(v(\tau),
 \tilde v^{\gamma N}(\tau))$ satisfies \eqref{7.1}, \eqref{7.2}.

ii) If $\tilde v^{\gamma N}(\tau_j^-)$ is known, then  on the segment
   $\Lambda_j=[\tau_j^-,\tau_j^+]$
we define $\tilde v^{\gamma N}$ as
$$
\tilde v^{\gamma N}(\tau)=U(\tilde\vv_k(\tau_j^-), \vv_k(\tau_j^-))\vv_k(\tau),\quad
k\le N.
$$

By applying It\^o's formula to the functional
$J(\tau)=\left(I_k(v(\tau))-I_k( \tilde v^{\gamma N}(\tau)\right)^2$ we  derive
 that if $J(\tau_j^+)=0$,
then $J(\tau)=0$ for all $\tau\in\Delta_j$ (see Lemma~7.1 in \cite{KP08}).
 Hence, the process
$\tilde v^{\gamma N}(\tau)$ is well defined for $\tau\in[\theta',T]$ and
\begin{equation}\label{h0}
I_k(v(\tau))\equiv I_k(\tilde v^{\gamma N}(\tau)), \quad k\le N.
\end{equation}
Let us abbreviate $U_k^j=(U(\tilde\vv_k(\tau_j^-),\vv_k(\tau_j^-))$.
Then on
 an interval $\Lambda_j$ the process $\tilde v^{\gamma N}$ satisfies the equation
\begin{equation}\label{h2}
d \tilde\vv_k^\gamma = U_k^j\left(\big(\nu^{-1}d\Psi_k(u)V(u)\big)^P+A_k^P(v)\right)\,d\tau
+\sum_l U_k^j\circ B_{kl}^P(v)\,d\bb_l(\tau).
\end{equation}
Letting formally $|\tilde\vv_k|/|\vv_k|=1$ if $\vv_k=0$,  we make the function
 $|\tilde\vv^\gamma_k|/|\vv_k|\equiv1$ along all trajectories.

Due to \eqref{7.2} and \eqref{h2},  $ \tilde v^{\gamma N}$ is an It\^o process
\begin{equation}\label{h3}
d\tilde \vv^\gamma_k=\hat A_k(\tau)\,d\tau +\sum\hat B_{kj}(\tau)\,d\bb_j(\tau),\quad
1\le k\le N.
\end{equation}
The coefficients $\hat A_k=\hat A^\gamma_k$ and $\hat B_{kj}=\hat B^\gamma_{kj}$ a.s.
 satisfy the estimates
\begin{equation}\label{h4}
|\hat A^\gamma(\tau)|\le\nu^{-1}C,\quad C^{-1}E\le \hat B^\gamma (\hat B^\gamma )^t
\le CE
\end{equation}
for all $\tau$, where $C$ depends only on $N$ and $P$ and we regard $\hat B^\gamma$ as
an $2N\times2N$-matrix.

Let us set
$$
\aA^\gamma_k(\tau)=\tilde\vv_{k}(\theta')+\int_{\theta'}^\tau\hat A^\gamma_k(s)\,ds,\quad
\MM^\gamma_k(\tau)=\sum_j\int_{\theta'}^\tau\hat B^\gamma_{kj}\,d\beta_j(\tau)
$$
(cf. \eqref{h1}) and consider the process
$$
\xi^\gamma(\tau)=(v^\gamma(\tau), \aA^\gamma(\tau), \MM^\gamma(\tau))
\in h^p\times \R^{2N}\times\R^{2N},\qquad \tau\ge\theta'.
$$
Then $\tilde v^{\gamma N}=\aA^\gamma(\tau)+\MM^\gamma (\tau)$ and due to \eqref{h4}
the family of  laws of the processes $\xi^\gamma$ is tight in the space
$C(\theta',T;h^p)\times C(\theta',T;\R^{2N})\times C(\theta',T;\R^{2N})$.
 Consider any limiting (as
$\gamma_j\to0$) law $\DD^0$ and find any process
 $(\bar v(\tau), \aA^0(\tau),\MM^0(\tau))$, distributed as $\DD^0$. Denote
$\tilde v^N(\tau)=\aA^0(\tau)+\MM^0(\tau)$ and consider the process
$(\bar v(\tau),\tilde v^N(\tau))\in h^p\times \R^{2N}$. It is easy to see that it
satisfies 1-3.  In
\cite{KP08} we show that  estimates \eqref{h4} imply that
$$
 \aA^0(\tau)= \int_{\theta'}^\tau B^N(s)\,ds,\quad
 \MM^0(\tau)=\int_{\theta'}^\tau a^N(s)\,dw(s),
$$
where $w(s)\in\R^{2N}$ is a standard Wiener process, while $B^N$ and $a^N$ meet \eqref{f2}.
That is,  $\tilde v^N(\tau)$ is an It\^o process
\begin{equation}\label{h7}
d \tilde v^N(\tau)=B^N(\tau)\,d\tau + a^N(\tau)\,dw(\tau),
\end{equation}
 where
\begin{equation}\label{h5}
|\hat B(\tau)|    \le C,\quad  C^{-1}E\le a^N  (a^N)^t(\tau)
\le CE\quad\forall\,\tau,\; a.s.
\end{equation}
These are the estimates, mentioned in item 4 above.

Now by \eqref{h4} and Theorem~4 from   Section~2.2 in \cite{Kry77},
 applied to the It\^o process $\tilde \vv_k$,  we have
\begin{equation}\label{h6}
\E\int^T_{\theta'} \chi_{\{I_k(v^\nu_P(\tau)) \le\delta   \}}d\tau\le C\delta,
 \quad \forall\,k\le N,
\end{equation}
where  $C=C(N,P)$.

Taking $\theta'=0$ and passing to a limit as $\nu\to0$ we see
 that the process $I_{Pk}(\tau)$
also
meets \eqref{h6}. Since $\DD(I_P(\cdot))\strela \DD(I(\cdot))$ as $P\to\infty$, then we
get estimate \eqref{e120}.

For any $\nu>0$ the processes $I_P^\nu$ and $I^\nu$ coincide on the event
$\{\sup_\tau|I^\nu(\tau)|_{h^I_p}\le P\}$. Due to \eqref{xxx} probability of this event goes to 1
 as $P\to\infty$, uniformly in $\nu$. So \eqref{h6} also implies that
\begin{equation}\label{h77}
\E\int^T_0 \chi_{ \{I_k^\nu(\tau)\le\delta\}}  \to0 \quad {\rm as} \; \delta\to0,
\end{equation}
uniformly in $\nu$.

\bibliography{meas}
\bibliographystyle{amsalpha}
\end{document}